\documentclass[11pt]{article}
\usepackage[a4paper]{geometry}
\usepackage{amsmath,amssymb,amsthm,amsfonts,mathtools,amsbsy}
\usepackage{xcolor}
\usepackage{enumitem}
\usepackage{hyperref}
\hypersetup{colorlinks=true,linkcolor=blue!60!black,filecolor=purple!50!black,citecolor=blue!50!green}
\usepackage{url}
\usepackage{accents}
\usepackage[noabbrev,nameinlink,capitalise]{cleveref}
\crefname{lemma}{Lemma}{Lemmas}
\crefname{corollary}{Corollary}{Corollaries}
\crefname{theorem}{Theorem}{Theorems}
\crefname{assumption}{Assumption}{Assumptions}
\crefname{equation}{}{}
\usepackage{nicefrac}
\usepackage{cite}
\usepackage{xurl}
\usepackage{tikz}
\usepackage{verbatim}
\usetikzlibrary{shapes,arrows,positioning,fit,calc}
\usepackage{todonotes}
\newenvironment{eqs} %
 { \begin{equation} \begin{aligned} } %
 { \end{aligned} \end{equation} \ignorespacesafterend } %

\theoremstyle{definition}
\theoremstyle{plain}
\makeatletter
\@ifclassloaded{beamer}{}{%
\newtheorem{theorem}{Theorem}[section]
\newtheorem{lemma}[theorem]{Lemma}

\crefname{assumption}{Assumption}{Assumptions}
}
\makeatother
\newtheorem*{remark*}{Remark}

\newtheorem*{acknowledgment*}{Acknowledgment}

\newcommand{\Bhnorm}[1]{\Vert #1 \Vert_{B_h}}

\newcommand{\CA}{C_A}
\newcommand{\Ca}{C_a}
\newcommand{\ca}{c_a}
\newcommand{\cA}{c_A}
\newcommand{\CD}{C_*}
\newcommand{\CT}{C_\urcorner}
\newcommand{\cAP}{c_{A_0}}
\newcommand{\const}[1]{C_{\eqref{#1}}}

\newcommand{\bpm}{\begin{pmatrix}}
\newcommand{\epm}{\end{pmatrix}}

\newcommand{\Vsh}{V_{\!\star \hspace*{-0.2mm}h}}
\newcommand{\Wsh}{W_{\!\star \hspace*{-0.2mm}h}}

\newcommand{\dgnorm}[1]{\norm{#1}_{V_h}}
\newcommand{\tdgnorm}[1]{{\Vert#1\Vert}_{\Vsh}}
\newcommand{\dgonorm}[1]{\norm{#1}_{V_h,\omega}}
\newcommand{\tdgonorm}[1]{{\Vert#1\Vert}_{\Vsh,\omega}}

\newcommand{\jmp}[1]{[\![ #1 ]\!]}

\newcommand{\avg}[1]{\{\!\!\{#1\}\!\!\}}
\newcommand{\dom}{\Omega} 

\newcommand{\Th}{{\calT_h}}
\newcommand{\Fh}{{\calF_h}}
\newcommand{\Fhi}{{\calF_h^i}}
\newcommand{\Fhb}{{\calF_h^b}}

\newcommand{\AK}{{A_K}}
\newcommand{\ILh}{\mathbb{L}}

\newcommand{\Lh}{\mathbb{L}(\Th)}
\newcommand{\LKh}{\mathbb{L}(K)}

\newcommand{\QKh}{Q_h(K)}

\newcommand{\Qh}{Q_h}
\newcommand{\ATh}{A_{\Th}}
\newcommand{\Bh}{{B_h}}

\renewcommand{\Re}{\operatorname{Re}}
\renewcommand{\Im}{\operatorname{Im}}

\DeclareMathOperator{\tc}{T}
\DeclareMathOperator{\tcq}{T_\IL}
\DeclareMathOperator{\tct}{T_\IT}
\DeclareMathOperator{\tcb}{T_{B}}
\DeclareMathOperator{\tch}{T_h}

\DeclareMathOperator{\grad}{\nabla}

\DeclareMathOperator{\diam}{diam}

\newcommand{\inner}[1]{( #1 )}

\newcommand*{\snorm}[1]{\left|#1\right|}
\newcommand*{\norm}[1]{\left\|#1\right\|}

\newcommand\restr[2]{{ \left.\kern-\nulldelimiterspace #1 \vphantom{\big|} \right|_{#2} }}

\newcommand{\IC}{\mathbb{C}}

\newcommand{\IH}{\mathbb{H}}

\newcommand{\IL}{\mathbb{L}}

\newcommand{\IP}{\mathbb{P}}

\newcommand{\IR}{\mathbb{R}}

\newcommand{\IT}{\mathbb{T}}
\newcommand{\ITh}{\mathbb{T}}

\newcommand{\calF}{\mathcal{F}}

\newcommand{\calH}{\mathcal{H}}

\newcommand{\calT}{\mathcal{T}}

\definecolor{pscol}{rgb}{0.8,0,0}

\usepackage{pgfplots} 
\pgfplotsset{compat=newest}
\usepgfplotslibrary{colorbrewer,groupplots}
\pgfplotsset{
	discard if not/.style 2 args={
	x filter/.append code={
	\edef\tempa{\thisrow{#1}}
	\edef\tempb{#2}
	\ifx\tempa\tempb
	\else
	
	\fi
	}
	}
}
\usepackage{pgfplotstable} 
\usepackage{booktabs} 
\usepackage{colortbl}
\makeatletter
\pgfplotstableset{
	discard if not/.style 2 args={
	row predicate/.append code={
	\def\pgfplotstable@loc@TMPd{\pgfplotstablegetelem{##1}{#1}\of}
	\expandafter\pgfplotstable@loc@TMPd\pgfplotstablename
	\edef\tempa{\pgfplotsretval}
	\edef\tempb{#2}
	\ifx\tempa\tempb
	\else
	\pgfplotstableuserowfalse
	\fi
	}
	}
}
\makeatother

\pgfplotscreateplotcyclelist{paulcolors1}{%
	orange,every mark/.append style={solid,fill=orange},mark=square*,very thick,mark size=3pt\\%
	teal,densely dashed,every mark/.append style={solid,fill=teal},mark=*,very thick,mark size=3pt\\%
}
\pgfplotscreateplotcyclelist{paulcolors2}{%
	orange,every mark/.append style={solid,fill=orange},mark=square*,very thick,mark size=3pt\\%
	orange,densely dashed,every mark/.append style={solid,fill=orange},mark=*,very thick,mark size=3pt\\%
	teal,every mark/.append style={solid,fill=teal},mark=square*,very thick,mark size=3pt\\%
	teal,densely dashed,every mark/.append style={solid,fill=teal},mark=*,very thick,mark size=3pt\\%
    violet,every mark/.append style={solid,fill=violet},mark=square*,very thick,mark size=3pt\\%
    violet,densely dashed,every mark/.append style={solid,fill=violet},mark=*,very thick,mark size=3pt\\%
}

\pgfplotscreateplotcyclelist{paulcolors3}{%
	orange,every mark/.append style={solid,fill=orange},mark=square*,very thick,mark size=3pt\\%
	orange,densely dashed,every mark/.append style={solid,fill=orange},mark=*,very thick,mark size=3pt\\%
	orange,dotted,every mark/.append style={solid,fill=orange},mark=diamond*,very thick,mark size=3pt\\%
	teal,every mark/.append style={solid,fill=teal},mark=square*,very thick,mark size=3pt\\%
	teal,densely dashed,every mark/.append style={solid,fill=teal},mark=*,very thick,mark size=3pt\\%
	teal,dotted,every mark/.append style={solid,fill=teal},mark=diamond*,very thick,mark size=3pt\\%
	violet,every mark/.append style={solid,fill=violet},mark=square*,very thick,mark size=3pt\\%
	violet,densely dashed,every mark/.append style={solid,fill=violet},mark=*,very thick,mark size=3pt\\%
	violet,dotted,every mark/.append style={solid,fill=violet},mark=diamond*,very thick,mark size=3pt\\%
}

\pgfplotscreateplotcyclelist{paulcolors4}{%
	orange,every mark/.append style={solid,fill=orange},mark=square*,very thick,mark size=3pt\\%
	teal,every mark/.append style={solid,fill=teal},mark=*,very thick,mark size=3pt\\%
	violet,every mark/.append style={solid,fill=violet},mark=diamond*,very thick,mark size=3pt\\%
    cyan,every mark/.append style={solid,fill=cyan},mark=star,very thick,mark size=3pt\\%
}

\title{Embedded Trefftz DG method for the Helmholtz equation}
\author{ 
Paul Stocker\thanks{Faculty of Mathematics, University of Vienna, Austria} 
\and 
Igor Voulis\thanks{Institut für Numerische und Angewandte Mathematik, Georg-August Universität Göttingen, Germany}
}

\date{}

\begin{document}
\maketitle

\begin{abstract}
We study an embedded Trefftz discontinuous Galerkin method for the Helmholtz equation.
The method starts from a polynomial DG space and enforces the Trefftz property through local constraints, avoiding an explicit construction of Trefftz basis functions.
For the global coupling we use a simple symmetric interior penalty DG bilinear form.
Since the resulting formulation is not coercive, stability is proved by a $T$-coercivity argument combined with a Schatz-type duality technique.
This yields wavenumber-explicit stability, quasi-optimality, and convergence estimates in standard DG norms under an explicit mesh resolution condition.

\medskip
\noindent\textbf{Keywords}: 
Helmholtz equation,
discontinuous Galerkin,
embedded Trefftz method

\medskip
\noindent\textbf{MSC2020}:
65N30, 
65N15, 
65N12,
35J05
\end{abstract}

\section{Introduction}
We present a wavenumber-explicit convergence analysis of an embedded Trefftz discontinuous Galerkin (DG) method for the Helmholtz equation, based on a simple symmetric interior penalty DG (SIPDG) formulation.
We prove stability and quasi-optimality by a Schatz-type duality argument under an explicit resolution condition.

DG methods provide a flexible framework for wave problems, allowing for local approximation spaces, nonmatching meshes, and robust handling of boundary conditions \cite{ABCM02}.
Trefftz methods are motivated by the desire to reduce the number of degrees of freedom using local constraints to select a PDE-adapted approximation space.
There are several classical examples. The Ultra Weak Variational Formulation (UWVF) \cite{CD98,BM08} and
Trefftz DG methods with plane-wave bases (often termed TDG/PWDG) yield highly compact discretizations and have been analyzed in mesh-skeleton norms via nonstandard inverse and approximation estimates tailored to plane waves \cite{GHP09,HMP11,HMP16,HMP16pwdg}.
Compared to classical Trefftz DG/UWVF/PWDG approaches, the embedded formulation permits a standard DG-style analysis on an underlying polynomial space, with the Trefftz property enforced through local constraints rather than through a bespoke non-polynomial skeleton calculus.

In this work we study an \emph{embedded Trefftz DG} method for the Helmholtz problem.
The starting point is a standard polynomial DG space on a mesh $\Th$, together with a local constraint mechanism that selects a PDE-adapted subspace, as introduced in \cite{LS_IJMNE_2023,lozinski19}.
The method avoids the explicit construction of local Trefftz bases and instead represents the Trefftz space through an embedding into the full polynomial DG space.
This approach was already tested numerically for several problems, including Helmholtz, in \cite{LS_IJMNE_2023}, Stokes in 
\cite{LLS_NM_2024}
unfitted problems in \cite{HLSW_M2AN_2022,S23}, 
parabolic problems in \cite{GPS_JSC_2025,H24},
and placed into a general abstract analysis framework in \cite{LLSV_ARXIV_2024}.

The SIPDG formulation we consider is deliberately simple and differs from the stronger stabilization mechanisms used in other Helmholtz discretizations, such as the UWVF and generalized DG formulations \cite{MPS13,CD98,M11} and strongly stabilized IPDG schemes \cite{FW09,FW11}, which build unconditional well-posedness into the bilinear form through additional derivative penalties and complex stabilization.
In \cite{MPS13}, resolution conditions such as $\omega h/\sqrt{p}$ sufficiently small and $\log(\omega)/p$ sufficiently small, stemming from the analysis in \cite{MS10,MS11}, are required for quasi-optimality of $hp$-UWVF/generalized DG methods.
As we only rely on a simple Schatz-type argument and no additional stabilization, our results come with a classical, more restrictive mesh condition \cite{MM95}, even for well-posedness, but they yield a transparent and fully explicit constant-tracking theory for the \emph{embedded Trefftz} paradigm.

For conforming finite element methods, stability and convergence typically require wavenumber-dependent resolution conditions and duality arguments of Schatz type \cite{Schatz74,IB95,IB97,MM95}, with refined analyses available 
for a sharper resolution condition in the $hp$-setting \cite{MS10,MS11}, and
for heterogeneous or piecewise smooth coefficients \cite{BCM25}.
In this paper we follow the same general philosophy: we derive wavenumber-explicit stability and error estimates under an explicit resolution condition.
Concretely, we require
\begin{equation}\label{eq:intro_resolution}
(1+\omega^2)\,h \le C_\Omega,
\end{equation}
with a domain-dependent constant $C_\Omega$.

In contrast to Trefftz methods, Hybridizable DG (HDG) methods reduce the global coupling to a skeleton space.
They have also been analyzed for the Helmholtz problem; see, for example, \cite{CGL09,GM11,CLX13}.
An in-depth comparison on general meshes between the Trefftz approach, HDG, and other methods is given in \cite{LSZ_PAMM_2024}.
While performance depends strongly on the number of faces per element and on the operator considered, the embedded Trefftz approach is competitive, especially for a higher number of facets per element.
We also point out that it goes far beyond the reduction achieved by condensing bubble functions in a standard DG method.

\paragraph{Contributions.}
We view the present work as establishing a baseline theory for embedded Trefftz-DG discretizations of Helmholtz: we analyze a simple global DG formulation and prove $T$-coercive stability and quasi-optimality by a Schatz-type argument under \eqref{eq:intro_resolution}.
\begin{itemize}
\item We apply a wavenumber-explicit Schatz-type duality argument to the embedded Trefftz DG method for Helmholtz, building on the abstract embedded framework of \cite{LLSV_ARXIV_2024}, and obtain stability in the $T$-coercive setting under \eqref{eq:intro_resolution}. Beyond its role in the present proof, this is of independent interest: it illustrates how stability properties can be transferred to a Trefftz-constrained space.
\item We derive wavenumber-explicit stability and quasi-optimality estimates in standard DG norms and combine them with polynomial $hp$ best-approximation (via a suitable (Trefftz) embedding mechanism), yielding a convergence theory for piecewise polynomial embedded Trefftz spaces (again under an explicit resolution condition).
\item The analysis is fully explicit in all constants, providing a transparent baseline theory for the embedded Trefftz paradigm and a clear roadmap for further developments. In particular, it paves the road to reach more demanding high-frequency regimes: improved fluxes (e.g.\ UWVF-type couplings), stronger stabilization, or refined local constraint spaces.
\end{itemize}

\paragraph{Outline.}
The remainder of the paper is organized as follows.
In \cref{sec:prelim} we introduce the abstract framework from \cite{LLSV_ARXIV_2024} and recall the Helmholtz problem and some standard inequalities.
In \cref{sec:etdg} we present the embedded Trefftz DG formulation for Helmholtz and verify the assumptions of the abstract framework, showing in particular the well-posedness of the formulation, wavenumber-explicit stability and quasi-optimality estimates under the resolution condition \eqref{eq:intro_resolution}.
Finally, in \cref{sec:numerics} we present numerical experiments that illustrate the theoretical findings.

\section{Preliminaries}\label{sec:prelim}
We will work in the abstract framework of \cite{LLSV_ARXIV_2024} to study the embedded Trefftz DG method for the Helmholtz equation. In this section we will first recall the main results of the abstract framework and then introduce the Helmholtz problem. 
\subsection{Abstract framework}

We revisit the abstract framework for embedded Trefftz DG methods from \cite{LLSV_ARXIV_2024}, which provides a general setting for analyzing such methods and deriving stability and quasi-optimality estimates. The proofs from \cite{LLSV_ARXIV_2024} can be adapted to to a complex-valued setting. For completeness, we provide the proofs of the main results in the appendix.

Let $V$ and $W$ be Hilbert spaces. We typically think about a Sobolev space on an open bounded Lipschitz domain $\Omega\subset \IR^d$ with $d=2,3$ (where a PDE problem may be posed). 
We consider the following abstract problem: Find $u\in V$ such that
\begin{equation}\label{eq:abstract} 
    \addtocounter{equation}{1} \tag{$\texttt{PDE}|\theequation$}
    a(u,v) = \ell(v) \quad \forall v\in W.
\end{equation}

Let $\Th$ be a mesh of $\Omega$. We fix a finite-dimensional trial space $V_h$, local spaces $\QKh$ for $K\in\Th$, and the associated global space
$
\Qh := \bigoplus_{K\in\Th}\QKh.
$
For each element $K\in\Th$ let $\AK:V_h\to\QKh'$ be a local linear operator and let $\ell_K\in\QKh'$ be a local right-hand side. 
We also consider a discrete sesquilinear form $a_h:(V+V_h)\times V_h\to\IC$ and a discrete right-hand side $\ell_h\in V_h'$.
We consider the embedded Trefftz space
\begin{align*}
    \ITh(K)&= \{v_h \in \IP^p(K) \mid \inner{\AK  v_h, q_h}_K=0,\ \forall q_h\in \QKh \}.
\end{align*}
The embedded Trefftz DG method reads
\begin{equation}\label{eq:PDEh}
\addtocounter{equation}{1} \tag{$\texttt{PDE}^{\IT}_h|\theequation$}
\begin{alignedat}{3}
    \text{Find } u_h \in V_h 
    \text{ such that } &&\quad\inner{\AK  u_h, q_h}_K&=\ell_K(q_h) &&\quad \forall q_h\in \QKh, K\in\Th,
    \\ \text{ and }  && \quad a_h(u_h,v_{\IT})&=\ell_h(v_{\IT}) &&\quad \forall v_{\IT}\in \tch \ITh,
\end{alignedat}
\end{equation}
for a suitable linear operator $\tch:\IT \to V_h$.
We introduce the bilinear form of the left-hand side of \eqref{eq:PDEh} as $\Bh:V_h\times Z_h\to\IC$, for the product test space $Z_h := \Qh \times \tch\ITh$, by
\[
  \Bh(u_h,(q_h,v_h)) := \sum_{K\in\Th}\langle \AK u_h,q_K\rangle + a_h(u_h,v_h).
\]
The following theorem provides well-posedness and a general C\'ea-type error estimate for the embedded Trefftz DG method, which will be the basis for our analysis of the Helmholtz problem.

\newcommand{\cea}{\rm{cea}}
\begin{theorem}[{\cite[\S 2.5]{LLSV_ARXIV_2024}}] \label{cor:cea}
Consider $\Vsh := V_h + V$ with two norms $\dgnorm{\cdot}$ and $\tdgnorm{\cdot}$ such that
$\dgnorm{\cdot}\leq \tdgnorm{\cdot}$ on $\Vsh$ and such that $\dgnorm{\cdot}$ and $\tdgnorm{\cdot}$ are equivalent on $V_h$.
Denote by $\CD\geq 1$ a constant such that
\begin{equation}\label{eq:Cstar}
  \tdgnorm{v_h}\le \CD\,\dgnorm{v_h}\qquad\forall v_h\in V_h.
\end{equation}
Assume that each $v_h\in V_h$ admits a fixed (linear) splitting $v_h=v_\IL+v_\IT$ with $v_\IL\in\ILh$ and $v_\IT\in\ITh$.

Let the following assumptions hold:
\begin{enumerate}[label=\roman*)]
\item There exist spaces $\LKh \subset V_h$ such that the linear maps
$\AK : \LKh \to \QKh'$ satisfy
\begin{equation}\label{eq:A_coerc}
    \norm{\AK  u_h }_{\QKh'} \geq \cA \dgnorm{u_h}
    \quad \forall u_h \in \ILh(K),\ K\in\Th
\end{equation} and $\IL = \bigoplus_{K\in \Th} \IL(K)$.
\item Simultaneous continuity on $\Vsh$, i.e.
\begin{equation}\label{eq:A_cont}
    \sum_{K\in\Th} \norm{\AK  u }_{\QKh'}^2 \leq \CA^2 \tdgnorm{u}^2
    \quad \forall u\in \Vsh.
\end{equation}
\item There exists a uniformly bounded injective linear operator $\tch:\ITh \to V_h$ such that for all $u_h\in \ITh$
\begin{equation}\label{eq:ah_coerc}
        \Re a_h(u_h,\tch u_h) \geq \ca \dgnorm{u_h}^2,
        \qquad
        \dgnorm{\tch u_h}\leq \dgnorm{u_h}.
\end{equation}
\item We assume that $a_h(\cdot,\cdot)$ is defined on $\Vsh \times \tch \ITh$ and continuous in the sense that
\begin{equation}\label{eq:ah_cont}
    |a_h(u,v_h)| \leq \Ca \tdgnorm{u} \dgnorm{v_h}
    \quad \forall u\in \Vsh,\ \forall v_h\in \tch \ITh.
\end{equation}
Note that these assumptions imply that there exists a constant $\CT$ such that
\begin{equation}\label{eq:CT_cont}
    |a_h(u_\IL,\tch v_\IT)|
    \leq \CT \norm{\ATh u_\IL}_{\Qh'} \sqrt{\Re a_h(v_\IT,\tch v_\IT)}
    \quad \forall u_\IL\in \ILh,\ \forall v_\IT\in \ITh.
\end{equation}
Indeed, by \eqref{eq:ah_cont}, \eqref{eq:Cstar}, \eqref{eq:A_coerc}, and \eqref{eq:ah_coerc},
we have 
$
\CT \le \frac{\Ca\,\CD}{\cA\sqrt{\ca}}.
$
\end{enumerate}

Then \eqref{eq:PDEh} is well-posed.
Furthermore, there exists a linear injective operator $\tcb:V_h\to Z_h$, such that for all $u_h\in V_h$,
\begin{equation}
  \Re \Bh(u_h,\tcb u_h) \ge \ca(\dgnorm{u_\IT}^2 + \dgnorm{u_\IL}^2) \ge \frac{\ca}{2}\,\dgnorm{u_h}^2
\end{equation}
with
\[
\alpha:=\CT^2+\frac{\ca}{\cA^2},
\qquad
\norm{\tcb u_h}_{Z_h}^2 \leq 4\,\dgnorm{u_\IT}^2 + \CD^2\CA^2\alpha^2 \dgnorm{u_\IL}^2.
\]

Let $u\in V$ be the solution to \eqref{eq:abstract}, and let $u_h\in V_h$ be the solution to \eqref{eq:PDEh}. Then there holds
\begin{eqs}\label{eq:cea}
\dgnorm{v_h-u_h} \leq C_{\cea}^{(1)} \tdgnorm{v_h - u} + C_{\cea}^{(2)}\norm{\ATh u - \ell_\Th}_{\Qh'} + C_{\cea}^{(3)}\norm{a_h(u,\cdot)-\ell_h}_{\tch\ITh'}.\\
\end{eqs}
with the three constants defined by
\begin{equation*}
C_{\cea}^{(1)}:=2\CD\Big( \frac{\CA^2}{\cA^2} +\frac{\CA^2}{\cA^2}\frac{\Ca^2}{\ca^2}\CD^2 +2\frac{\Ca}{\ca} \Big),\
C_{\cea}^{(2)}:=2\CD\Big( \frac{\CA}{\cA^2}\frac{\Ca^2}{\ca^2}\CD^2 +\frac{\CA}{\cA^2} \Big),\
C_{\cea}^{(3)}:=4\CD\frac{1}{\ca}.
\end{equation*}
\end{theorem}

The following lemma allows us to verify the assumption \eqref{eq:A_coerc} based on a perturbation argument.
\begin{lemma}[see {\cite[Lemma 3.4]{LLSV_ARXIV_2024}}]
    \label{lem:abstractneumann}
    Let $A_{K,0}:\LKh \to \QKh'$ be invertible with 
    \begin{equation}\label{eq:A0_coerc}
        \norm{A_{K,0}  u_h }_{\QKh'} \geq \cAP \norm{u_h}_{V_h}
        \quad \forall u_h \in \ILh(K),\ K\in\Th.
    \end{equation}
    there exists a constant $\gamma \in (0,1)$ such that
    \begin{equation}\label{eq:nearlycontant}
        \norm{\AK  u - A_{K,0}u}_{\QKh'}\le\gamma\norm{A_{K,0}u}_{\QKh'} \quad \forall u\in \LKh,
    \end{equation}
    then the restriction $\AK : \LKh \to \QKh'$ is invertible, with
    \begin{equation}\label{eq:A_A0_coerc}
        \norm{A_{K}  u_h }_{\QKh'} \geq \cAP(1-\gamma)\, \norm{u_h}_{V_h}
        \quad \forall u_h \in \ILh(K),\ K\in\Th.
    \end{equation}
\end{lemma}

Finally, if one wishes for bounds in weaker norms, one can apply an Aubin--Nitsche-type duality argument, to obtain the following result.
\begin{theorem} \label{th:aubinnitsche}
Assume that the assumptions from \cref{cor:cea} hold and additionally that $\Vsh$ and $\Wsh$ are Hilbert spaces such that
$\Vsh\subset \IH$ and that there exists $C_{\IH}>0$ such that
\begin{equation}\label{eq:IH_DG_dom}
    \norm{v}_{\IH}\le C_{\IH}\dgnorm{v}\qquad\forall v\in\Vsh.
\end{equation}
Assume moreover that the norms $\dgnorm{\cdot}$ and $\tdgnorm{\cdot}$ are equivalent on $V_h$ (with a constant $\CD$ as in \eqref{eq:Cstar}).
Furthermore, assume the adjoint consistency
\begin{equation}
\label{eq:a-bit-consistent-exp}
a_h(v, z) = a(v,z) \quad \forall v \in \Vsh,\ \forall z \in \{y\in W\mid a(\cdot,y)\in \IH' \},
\end{equation}
and the adjoint continuity: there exists $C_{\rm adj}>0$ such that
\begin{equation}
\label{eq:a-bit-cont-exp}
|a_h(v, z)| \le C_{\rm adj}\,\tdgnorm{v}\,\norm{z}_{\Wsh}
\quad \forall v \in \Vsh,\ \forall z \in  W_h + \{y\in W\mid a(\cdot,y)\in \IH' \}.
\end{equation}
Assume, in addition, that the adjoint problem is well posed with $\IH$-data, i.e.,
for every $\phi\in \IH$ there exists (at least one) $z_\phi\in W$ such that
\begin{equation}\label{eq:adjoint-IH-wellposed}
a(v,z_\phi) = (\phi,v)_{\IH}\qquad\forall v\in V.
\end{equation}
Assume that there exist constants $h_H>0$, $C_{\rm reg}^{(1)}>0$, $C_{\rm reg}^{(2)}>0$ such that for all
$z\in W$ with $a(\cdot,z)\in \IH'$,
\begin{align}
\label{eq:H-reg-bound-exp}
\inf_{z_h \in \tch \ITh} \norm{z -z_h}_{ W_h}
&\le C_{\rm reg}^{(1)}\,h_H \sup_{v\in V} \frac{|a(v,z)|}{\norm{v}_{\IH}},
\\
\label{eq:H-reg-bound2-exp}
\inf_{w_h \in  W_h} \norm{z -w_h}_{\Wsh}
&\le C_{\rm reg}^{(2)}\,h_H \sup_{v\in V} \frac{|a(v,z)|}{\norm{v}_{\IH}}.
\end{align}

Let $u\in V$ solve \eqref{eq:abstract} and $u_h\in V_h$ solve \eqref{eq:PDEh}.
Using the notation for the three constants appearing in \eqref{eq:cea}
Then the $\IH$-error satisfies
\begin{equation}\label{eq:IH_bound_framework}
\begin{aligned}
\norm{u-u_h}_{\IH}
\le &
\frac{C_{\IH}}{\ca}\,\norm{a_h(u,\cdot)-\ell_h}_{\tch\ITh'}
\\
&\ +\ C_{\rm adj}\,(C_{\rm reg}^{(1)}+2C_{\rm reg}^{(2)})\,h_H
\Bigg[
\bigl(1+2\CD+\CD\,C_{\cea}^{(1)}\bigr)\inf_{v_h\in V_h}\tdgnorm{u-v_h}
\\
&
+\, \CD\,C_{\cea}^{(2)}\,\norm{\ATh u-\ell_\Th}_{\Qh'}
+\, \CD\Bigl(C_{\cea}^{(3)}+\frac1{\ca}\Bigr)\norm{a_h(u,\cdot)-\ell_h}_{\tch\ITh'}
\Bigg].
\end{aligned}
\end{equation}
\end{theorem}

\subsection{Helmholtz equation}
Let $\Omega\subset \IR^d$ be a bounded domain with Lipschitz boundary $\partial\Omega$.
We consider the Helmholtz equation
\begin{eqs}\label{eq:helmholtz}
    -\Delta u - \omega^2 u = f \quad \text{in } \Omega,\\
    \grad u \cdot n + i\omega u = g \quad \text{on } \partial\Omega, 
\end{eqs}
where $\omega>0$ is the wave number, $f$ is the source term and $g$ is the boundary data.
Variational formulation of the Helmholtz problem reads: Find $u\in H^1(\Omega)$ such that
\begin{eqs}\label{eq:helmholtzvar}
    \inner{\grad u, \grad v}_\Omega - \omega^2 \inner{u, v}_\Omega + i\omega\inner{u,v}_{\partial\Omega}
    =&\ \inner{f, v}_\Omega + \inner{g, v}_{\partial\Omega}\quad \forall v\in H^1(\Omega),
\end{eqs}
Let $\dom$ be a bounded Lipschitz domain. Then there is a constant $C(\dom,\omega)>0$ such that for all $f\in H^1(\dom)'$, $g\in H^{-1/2}(\partial\dom)$, then the variational problem \eqref{eq:helmholtzvar} has a unique solution $u\in H^1(\dom)$.
The norm we use for the analysis is given by
\[ \norm{u}_V^2 = \norm{\grad u}^2_\Omega + \omega^2\norm{u}^2_\Omega. \]

To discuss the non-conforming discretisation of the Helmholtz equation, we introduce the following standard DG notation. 
Let $\mathcal T_h$ be a partition of the bounded Lipschitz domain $\Omega\subset\mathbb{R}^d$ into non-overlapping Lipschitz elements $K$ with $\overline\Omega=\bigcup_{K\in\mathcal T_h}\overline K$.
Let $\Fh$ be the set of all faces in the mesh $\Th$, and let $\Fhi$ and $\Fhb$ be the sets of interior and boundary faces, respectively. Assume that $\Fh$ has finite $(d-1)$-dimensional Hausdorff measure.

For $F\in\Fhi$ shared by elements $K^+$ and $K^-$ with outward unit normal vector $n$ pointing from $K^+$ to $K^-$. 
We define the average and jump operators as
\begin{equation*}
    \avg{v} = \frac{1}{2}(v^+ + v^-), \quad \jmp{v} = v^+ - v^-,
\end{equation*}
where $v^\pm = v|_{K^\pm}$.

A consistent formulation of \eqref{eq:helmholtzvar} on the broken space $\Vsh(\Th) := \{v\in L^2(\Omega) : v|_K \in H^1(K), \nabla v|_{\partial K}\cdot n \in L^2(\partial K),\ K\in\Th\}$ is given by  the bilinear form 
\begin{eqs}
    a_h(u,v) =&\ \inner{\grad u, \grad v}_\Th - \omega^2 \inner{u, v}_\Th + i\omega\inner{u,v}_\Fhb\\
              &\ - \inner{\avg{\grad u \cdot n},\jmp{v}}_\Fhi - \inner{\jmp{u},\avg{\grad v \cdot n}}_\Fhi + \frac{\alpha p^2}{h}\inner{\jmp{u},\jmp{v}}_\Fhi
\end{eqs}
and right hand side 
\begin{equation}
    \ell_h(v) = \inner{f, v}_\Th + \inner{g, v}_{\Fhb},
\end{equation}
where we assume that $f\in L^2(\Omega)$ and $g\in L^2(\partial\Omega)$.
Here and in the following we use the shorthand notation
\[
    \inner{u,v}_\Th := \sum_{K\in\Th} \inner{u,v}_K, \quad \norm{u}^2_\Th := \sum_{K\in\Th} \norm{u}^2_K,
\]
where $\inner{u,v}_K$ and $\norm{u}_K$ denote the standard $L^2(K)$ inner product and norm, respectively.

On $\Vsh(\Th)$, the broken variants of the norm $\norm{\cdot}_V$ are given by
\begin{align*}
    \dgonorm{u}^2 & :=  \norm{\grad u}^2_\Th + \omega^2\norm{u_h}^2_\Th +\frac{p^2}{h}\norm{\jmp{u_h}}^2_\Fhi+\omega\norm{u_h}^2_\Fhb , \\  
    \tdgonorm{u}^2 &  := \dgonorm{u}^2 + \sum_{K\in\Th}p^{-2}h\norm{\grad u \cdot n}^2_{\partial K}.
\end{align*}

\subsection{Assumptions and inequalities}
We recall some standard inequalities and assumptions on the mesh $\Th$.
These will be used in the analysis, see e.g. \cite{CDG22,CDGH17}. %

We will work under the main assumption: 
\begin{itemize}
  \item[(A0)] the mesh $\Th$ is sufficiently fine with respect to the wavenumber $\omega$ in the sense that
\begin{equation}\label{ass:oh}
    (1+\omega^2)h \leq C_\Omega
\end{equation}
\end{itemize}
where $C_\Omega$ only depends on the domain.
Here $h = \sup_{K\in\Th} h_K$ with $h_K = \diam K$.

We assume the following uniform shape-regularity properties (with constants independent of $h,p$ and $\omega$):
\begin{itemize}
  \item[(A1)] For each $K\in\mathcal T_h$ there exist balls $b_K\subset K\subset B_K$ (concentric or not) such that
    \[
      \frac{\operatorname{diam}(B_K)}{\operatorname{diam}(b_K)}\leq \gamma_\star,
    \]
    for a uniform $\gamma_\star>0$. Moreover, $K$ is star-shaped with respect to the ball $b_K$.
  \item[(A2)]
    The boundary $\partial K$ can be partitioned into mutually disjoint subsets $\{F_i\}_{i=1}^{n_K}$
    with $n_K\leq n_{\max}$ uniformly and $\operatorname{diam}(K)\leq C_{\mathrm{bdry}}\operatorname{diam}(F_i)$ for each $i$.
    For every facet piece $F_i$ there exists a sub-element $K_{F_i}\subset K$ which is star-shaped
    with respect to a point $x_i^0\in K$ and whose local diameter satisfies $h_{K_{F_i}}\simeq h_K$ uniformly.
    Moreover there exists a uniform constant $c_{\mathrm{out}}>0$ such that
    \[
      (x-x_i^0)\cdot n_{F_i}(x)\ge c_{\mathrm{out}}\,h_K\qquad\forall x\in F_i,
    \]
    where $n_{F_i}$ denotes the outward unit normal on $F_i$.
  \item[(A3)] Each $\partial K$ is the union of a finite number of closed $C^1$ surfaces.
\end{itemize}

All constants $C$ below depend only on
$\gamma_\star,\ c_{\mathrm{out}},\ C_{\mathrm{bdry}},\ n_{\max},\ d$
and are independent of $h$ and $p$.
We will assume that $\Th$ satisfies \textup{(A1)}--\textup{(A3)}. 

\begin{lemma}[Scaled Poincar\'e--Friedrichs]\label{lem:pf}
Let $K\in\Th$. Then there exists a constant $\const{pf}>0$,
depending only on the shape-regularity parameters and $d$, such that for all $v\in H^1(K)$,
\begin{equation}\tag{pf}\label{pf}
\|v\|_{K}^2
\le \const{pf}\Big(h_K^2\|\nabla v\|_{K}^2 + h_K\|v\|_{\partial K}^2\Big).
\end{equation}
\end{lemma}

\begin{lemma}[Polynomial inverse trace]\label{lem:itr}
Let $K\in\Th$ and $u\in\IP^p(K)$. Then there exists $\const{itr}>0$ such that
\begin{equation}\tag{itr}\label{itr}
\|u\|_{\partial K} \leq \const{itr}\,p\,h_K^{-1/2}\,\|u\|_{K}.
\end{equation}
\end{lemma}

\begin{lemma}[Polynomial inverse inequality]\label{lem:inv}
Let $K\in\Th$ and $u\in\IP^p(K)$. Then there exists $\const{inv}>0$ such that
\begin{equation}\tag{inv}\label{inv}
  \|\nabla u\|_{K} \leq \const{inv}\,p^2\,h_K^{-1}\,\|u\|_{K}.
\end{equation}
\end{lemma}

\begin{lemma}[Continuous trace]\label{lem:tr}
Let $K\in\Th$ and let $F\subset\partial K$ be a (closed) facet piece. Then there exists $\const{tr}>0$ such that
for all $v\in H^1(K)$
\begin{equation}\tag{tr}\label{tr}
  \|v\|_{F}^2 \leq \const{tr}
  \Big( h_F^{-1}\|v\|_{K} + \|\nabla v\|_{K}\Big)\|v\|_{K}.
\end{equation}
\end{lemma}

\begin{lemma}[Polynomial norming-set inequality on an inscribed ball]\label{lem:norming}
Assume \textup{(A1)} and let $b_K\subset K$ be the ball from \textup{(A1)} with radius $r_K$.
Set
\[
\vartheta_K := \frac{r_K}{h_K}\in(0,1),
\qquad \vartheta := \inf_{K\in\Th}\vartheta_K>0.
\]
Then there exists a constant $\const{rem}\ge 1$, depending only on $\gamma_\star$ and $d$, such that
for all $p\ge 1$ and all $v\in \IP^p(K)$,
\begin{equation}\tag{rem}\label{rem}
\|v\|_{K} \le \const{rem}\,\vartheta^{-p}\,\|v\|_{b_K},
\qquad
\|\nabla v\|_{K} \le \const{rem}\,\vartheta^{-p}\,\|\nabla v\|_{b_K}.
\end{equation}
\end{lemma}

\section{Embedded Trefftz DG formulation}\label{sec:etdg}
We discretize \eqref{eq:helmholtzvar} using the global sesquilinear form $a_h(\cdot,\cdot)$ and the right-hand side $\ell_h(\cdot)$ defined above. On each element $K\in\Th$ we use the polynomial space $\IP^p(K)$ of degree at most $p\ge 0$, and we set
\[
V_h := \{v\in L^2(\Omega): v|_K\in \IP^p(K)\ \forall K\in\Th\}.
\]
To formulate the embedded Trefftz method we need, for each $K\in\Th$, a local space $\LKh$, a local test space $\QKh$, a local operator $\AK:V_h\to\QKh'$, and a local right-hand side $\ell_K\in\QKh'$. The (global) complement of the Trefftz space is
$
\ILh := \bigoplus_{K\in\Th}\LKh .
$
Once the operator $\AK$ is fixed, it induces the local weak Trefftz spaces
\[
\ITh(K) := \{v_h\in \IP^p(K): \inner{\AK v_h,q_h}_K = 0 \ \forall q_h\in\QKh\},
\]
and the corresponding global space
$
\ITh := \bigoplus_{K\in\Th}\ITh(K).
$
We now specify these objects.

We introduce the spaces
\begin{equation}\label{eq:LKQK}
    \LKh := (|x-x_K|^2-r_K^2)\,\IP^{p-2}(K),\qquad \QKh:=\IP^{p-2}(K),
\end{equation}
Equip $\QKh$ with
\[
    \|q_h\|_{\QKh}^2 := h_K^{2}\|\nabla q_h\|_{K}^2 + p^{2} h_K \|q_h\|_{\partial K}^2,
\qquad
\|q\|_{\Qh}^2 := \sum_{K\in\Th}\|q_K\|_{\QKh}^2,
\]
and let $\|\cdot\|_{\Qh'}$ denote the induced dual norm.

The embedded Trefftz DG method for the Helmholtz problem is then given by \eqref{eq:PDEh} with
\begin{equation}\label{eq:AK}
    \inner{A_{K} u_h, q_h}_K = h_K\inner{-\Delta u_h-\omega^2 u_h, q_h}_{K} \quad \forall q_h\in \QKh.
\end{equation}
and 
\begin{equation}\label{eq:lK}
    \ell_K(q_h) = h_K\inner{f, q_h}_K \quad \forall q_h\in \QKh.
\end{equation}

We want to emphasize that the first line in \eqref{eq:PDEh} does not enter the final linear system. 
In fact we can conveniently eliminate the degrees of freedom associated with $\LKh$ on each element $K\in\Th$ as the system decouples.
Instead, we follow the idea of the embedded Trefftz DG method \cite{LS_IJMNE_2023} and solve the problem in two steps, generating the Trefftz space along the way.
In order to solve the problem precompute the particular solution $u_{h,f}\in \ILh$ satisfying using an SVD and the pseudo-inverse to solve
\begin{equation}\label{eq:particular}
    \inner{A_{K} u_{h,f}, q_h}_K = \ell_K(q_h) \quad \forall q_h\in \QKh, K\in\Th.
\end{equation}
The kernel of $A_{K}$, computed using the SVD, gives a basis for the local Trefftz space $\ITh(K)$.
Then solve for $u_{h,0}\in\ITh$ satisfying
\begin{equation}\label{eq:homogeneous}
    a_h(u_{h,0},v_h) = \ell_h(v_h) - a_h(u_{h,f},v_h) \quad \forall v_h\in \ITh.
\end{equation}
The final solution is then given by $u_h = u_{h,0} + u_{h,f}$.
While we are aught to test with $\tch\ITh$ we note that our choice of $\tch$ below will map $\ITh$ to itself, so we can test with $\ITh$ instead.

In the following sections we verify the assumptions of \cref{cor:cea} for the embedded Trefftz DG method for the Helmholtz equation.
We start with the analysis of the local problem in \cref{sec:local_analysis}.
Then, in \cref{sec:global_analysis} we verify the assumptions on the global bilinear form $a_h(\cdot,\cdot)$.
Finally, in \cref{sec:quasi_best} we combine the results to obtain stability and quasi-optimality estimates for the embedded Trefftz DG method for Helmholtz.

\subsection{Analysis of the local problem}\label{sec:local_analysis}

We need to verify assumption \eqref{eq:A_cont} and \eqref{eq:A_coerc} for the operator defined in \eqref{eq:AK}.
To verify assumption \eqref{eq:A_coerc} we will use \cref{lem:abstractneumann}.
To this end we need to define a suitable prototype operator $A_{K,0}$ in \cref{sec:prototype} which we will show to be invertible with suitable bounds.
Then, in \cref{sec:AK} we will show that the operator defined in \eqref{eq:AK} is a small perturbation of the prototype operator.

\subsubsection{Prototype operator}\label{sec:prototype}
We can now introduce the prototype operator $A_{K,0}:\LKh \to \QKh'$ defined by
\begin{equation}\label{eq:AK0}
    \inner{A_{K,0} u_h, q_h}_K := \inner{-h_K\Delta u_h, q_h}_{K} \quad \forall q_h\in \QKh.
\end{equation}

\begin{lemma}\label{lem:nonharmonicNE}
Assume \textup{(A0)}--\textup{(A3)}.
Then there exist constants $C_{A,0}>0$ and $c_{A,0}>0$, independent of $h$ and $\omega$, such that
\begin{equation}\label{eq:AK0_scaled_bounds}
\|A_0 u\|_{\Qh'} \le C_{A,0}\tdgonorm{u}
\qquad\forall u\in \Vsh \cap H^2(\Th).
\end{equation}
and
\begin{equation}\label{eq:AK0_scaled_coerc}
\|A_0 u_h\|_{\Qh'} \ge c_{A,0}\dgonorm{u_h}
\qquad\forall u_h\in \Lh.
\end{equation}
Moreover, one may take $C_{A,0}=\sqrt 2$, and
\[
c_{A,0}
=
\frac{1}{\sqrt{C_\star}}
\frac{\vartheta^{p}}{\sqrt{\const{inv}^2+\const{itr}^2}\const{rem}}
\frac{1}{p^{3}},
\]
where $C_\star>0$ depends only on the shape-regularity parameters from \textup{(A1)}--\textup{(A3)}, on $d$,
on $\const{tr}$, and on $C_\Omega$, but is independent of $h$, $p$, and $\omega$.
\end{lemma}

\begin{proof}
\textbf{Continuity.}
Fix $K\in\Th$ and $u\in H^2(K)$. For any $q_h\in\QKh$, integration by parts yields
\[
h_K\langle-\Delta u,q_h\rangle_K
= h_K\langle\nabla u,\nabla q_h\rangle_K - h_K\langle\partial_n u,q_h\rangle_{\partial K}.
\]
By definition of $\|\cdot\|_{\QKh}$,
$\|\nabla q_h\|_K\le h_K^{-1}\|q_h\|_{\QKh}$ and
$\|q_h\|_{\partial K}\le (p\sqrt{h_K})^{-1}\|q_h\|_{\QKh}$, hence
\[
|h_K\langle-\Delta u,q_h\rangle_K|
\le \Big(\|\nabla u\|_K + \frac{\sqrt{h_K}}{p}\|\partial_n u\|_{\partial K}\Big)\,\|q_h\|_{\QKh}.
\]
Taking the supremum over $q_h\neq0$ gives
\[
\|A_{K,0}u\|_{\QKh'}
\le \|\nabla u\|_K + \frac{\sqrt{h_K}}{p}\|\partial_n u\|_{\partial K}.
\]
Using $(a+b)^2\le 2(a^2+b^2)$ and summing over $K$,
\[
\|A_0 u\|_{\Qh'}^2
\le 2\|\nabla u\|_\Th^2 + 2\sum_{K\in\Th}\frac{h_K}{p^2}\|\nabla u\cdot n\|_{\partial K}^2
\le 2\,\tdgonorm{u}^2,
\]
so \eqref{eq:AK0_scaled_bounds} holds with $C_{A,0}=\sqrt 2$.

\textbf{Lower bound.}
Let $u_h\in\LKh$ and choose $q_h=-\Delta u_h\in\QKh$:
\begin{equation}\label{eq:dual_lap_scaled}
\|A_{K,0}u_h\|_{\QKh'}
\ge \frac{h_K\|\Delta u_h\|_K^2}{\|\Delta u_h\|_{\QKh}}.
\end{equation}
Using \eqref{inv} and \eqref{itr} for $r_h\in\IP^{p-2}(K)$,
\[
    \|r_h\|_{\QKh}^2
=h_K^2\|\nabla r_h\|_K^2+p^2h_K\|r_h\|_{\partial K}^2
\le (\const{inv}^2+\const{itr}^2)p^4\,\|r_h\|_K^2,
\]
hence applying this with $r_h=\Delta u_h$ in \eqref{eq:dual_lap_scaled} gives
\begin{equation}\label{eq:dual_ge_Delta_scaled}
\|A_{K,0}u_h\|_{\QKh'}
\ge
\frac{h_K}{\sqrt{\const{inv}^2+\const{itr}^2}}\frac{1}{p^2}\,\|\Delta u_h\|_K.
\end{equation}

As before, $u_h|_{\partial b_K}=0$, hence $u_h|_{b_K}\in H^2(b_K)\cap H_0^1(b_K)$.
Using the standard Dirichlet estimate on the ball (dimensionless constant $C_{\rm ball}$ depending only on $d$),
together with \eqref{rem} and $r_K=\vartheta_K h_K\ge \vartheta h_K$, yields
\begin{equation}\label{eq:Delta_controls_grad_u_again}
h_K\,\|\Delta u_h\|_{K}
\ge
\frac{\vartheta^{p+1}}{\const{rem}\,C_{\rm ball}}
\Big(\|\nabla u_h\|_{K}+h_K^{-1}\|u_h\|_{K}\Big).
\end{equation}
Combining \eqref{eq:dual_ge_Delta_scaled} and \eqref{eq:Delta_controls_grad_u_again} yields
\begin{equation}\label{eq:A_controls_bulk_again}
\|A_{K,0}u_h\|_{\QKh'}
\ge
\underbrace{\frac{\vartheta^{p+1}}{\sqrt{\const{inv}^2+\const{itr}^2}\,\const{rem}\,C_{\rm ball}}}_{=: \beta_p}\,
\frac{1}{p^2}\,
\Big(\|\nabla u_h\|_{K}+h_K^{-1}\|u_h\|_{K}\Big).
\end{equation}

By \eqref{tr} on each facet piece $F\subset\partial K$ and Young's inequality with parameter $h_K$,
\[
\|u_h\|_{F}^2
\le \const{tr}\Big(h_F^{-1}\|u_h\|_{K} + \|\nabla u_h\|_{K}\Big)\|u_h\|_{K}
\le C\,\Big(h_K^{-1}\|u_h\|_K^2 + h_K\|\nabla u_h\|_K^2\Big),
\]
with $C$ depending only on $\const{tr}$ and the comparability $h_F\simeq h_K$ from \textup{(A2)}.
Summing over all facet pieces of $\partial K$ and using $n_K\le n_{\max}$ gives
\begin{equation}\label{eq:trace_sq_global}
\|u_h\|_{\partial K}^2 \le C\,\Big(h_K^{-1}\|u_h\|_K^2 + h_K\|\nabla u_h\|_K^2\Big),
\end{equation}
with $C$ depending only on $\const{tr}$ and the shape parameters.

Using $\|[u_h]\|_F^2\le 2(\|u_h^+\|_F^2+\|u_h^-\|_F^2)$, \eqref{eq:trace_sq_global}, and \eqref{ass:oh}
(which implies $\omega h_K\le C_\Omega$ and $\omega h_K^{-1}\le C_\Omega\,h_K^{-2}$), we obtain for $u_h\in\Lh$:
\[
\frac{p^2}{h}\|[u_h]\|_{\Fhi}^2 + \omega\|u_h\|_{\Fhb}^2
\le C\sum_{K\in\Th}\Big(
p^2 h_K^{-2}\|u_h\|_K^2 + p^2\|\nabla u_h\|_K^2
+ h_K^{-2}\|u_h\|_K^2 + \|\nabla u_h\|_K^2
\Big).
\]
Since $p\ge 1$ and $\omega^2\|u_h\|_K^2\le C_\Omega^2\,h_K^{-2}\|u_h\|_K^2$, this yields
\begin{equation}\label{eq:DG_by_bulk_improved}
\dgonorm{u_h}^2
\le C_{\rm DG}\,p^2 \sum_{K\in\Th}\Big(\|\nabla u_h\|_K^2 + h_K^{-2}\|u_h\|_K^2\Big),
\end{equation}
with $C_{\rm DG}$ depending only on the shape parameters, $\const{tr}$ and $C_\Omega$.

Finally, from \eqref{eq:A_controls_bulk_again} and $(a+b)^2\le 2(a^2+b^2)$ we have
\[
\|\nabla u_h\|_K^2 + h_K^{-2}\|u_h\|_K^2
\le 2\Big(\|\nabla u_h\|_K + h_K^{-1}\|u_h\|_K\Big)^2
\le \frac{2p^4}{\beta_p^2}\,\|A_{K,0}u_h\|_{\QKh'}^2.
\]
Insert this into \eqref{eq:DG_by_bulk_improved}, sum over $K$, and use $\beta_p\sim \vartheta^{p}/((\const{inv}^2+\const{itr}^2)^{1/2}\const{rem})$
to get
\[
\dgonorm{u_h}^2 \le C_\star
\frac{(\const{inv}^2+\const{itr}^2)\,\const{rem}^2}{\vartheta^{2p}}
p^{6}
\|A_0u_h\|_{\Qh'}^2.
\]
Taking square roots and rearranging yields \eqref{eq:AK0_scaled_coerc} with the stated $c_{A,0}$.
\end{proof}

We emphasize that we do not expect the exponential dependence of $c_{A,0}^{-1}$ on $p$ in \cref{lem:nonharmonicNE} to be sharp.
In the present proof, this dependence enters through the inequality \eqref{rem}, which introduces the factor $\vartheta^{-p}$ when passing from estimates on the inscribed ball $b_K$ to estimates on the full element $K$.
Within the current argument, this step appears to be unavoidable, so improving the $p$-dependence would likely require a different local coercivity argument.
On the other hand, the numerical results in \cref{sec:pconvergence} do not indicate such a severe deterioration: they rather suggest at most algebraic decay of $c_{A,0}$ in $p$, and the method still exhibits the expected exponential convergence in $p$ for smooth solutions.
Clarifying the sharp $p$-dependence of $c_{A,0}$ is therefore left for future work.

\subsubsection{Properties of the local operator \texorpdfstring{$A_{K}$}{AK}}\label{sec:AK}
Now that we have set up a prototype operator, we are ready to proof \eqref{eq:A_coerc} for our local operator given in \eqref{eq:AK}.
The continuity property \eqref{eq:A_cont} follows as for the prototype operator in the proof of \cref{lem:nonharmonicNE}.
Thus we turn towards the lower bound, which we achieve by using our prototype operator and applying \cref{lem:abstractneumann}.

\begin{lemma}\label{lem:loccoerc}
Assume \textup{(A0)}--\textup{(A3)}. 
Let $c_{A,0}$ be the constant provided by \cref{lem:nonharmonicNE}.
Define
\begin{equation}\label{eq:gamma_loccoerc}
    \gamma_K := \frac{2\sqrt{\const{pf}}}{c_{A,0}}\omega h_K,
\qquad
\gamma := \sup_{K\in\Th}\gamma_K .
\end{equation}
If $\gamma<1$ (equivalently $\omega h < \frac{c_{A,0}}{2\sqrt{\const{pf}}}$, or using
\eqref{ass:oh} take $h < \frac{c_{A,0}^2}{4\const{pf}C_\Omega}$), then $\AK:\LKh\to\QKh'$
is invertible for every $K\in\Th$ and
\begin{equation}\label{eq:loccoerc}
\|\AK u_h\|_{\QKh'} \ge \cA\,\dgonorm{u_h}
\qquad \forall u_h\in\LKh,\ \forall K\in\Th,
\end{equation}
with the explicit constant $\cA:=c_{A,0}(1-\gamma)$.
\end{lemma}
\begin{proof}
For any $u_h\in\LKh$ and any $q_h\in\QKh$,
\[
\langle (\AK-A_{K,0})u_h, q_h\rangle_K = -\omega^2 h_K\langle u_h,q_h\rangle_K.
\]
Hence, by definition of the dual norm and Cauchy--Schwarz,
\[
\|(\AK-A_{K,0})u_h\|_{\QKh'}
= \sup_{q_h\in\QKh\setminus\{0\}}
\frac{\omega^2 h_K|\langle u_h,q_h\rangle_K|}{\|q_h\|_{\QKh}}
\le \omega^2 h_K \|u_h\|_K \sup_{q_h\neq0}\frac{\|q_h\|_K}{\|q_h\|_{\QKh}}.
\]
Using \eqref{pf} and the definition of $\|\cdot\|_{\QKh}$,
\[
\|q_h\|_K
\le \sqrt{\const{pf}}\Big(h_K\|\nabla q_h\|_K + h_K^{1/2}\|q_h\|_{\partial K}\Big)
\le 2\sqrt{\const{pf}}\,\|q_h\|_{\QKh},
\]
for $p\ge1$. Therefore
\begin{equation}\label{eq:perturb_loccoerc}
    \|(\AK-A_{K,0})u_h\|_{\QKh'} \le 2\sqrt{\const{pf}}\,\omega^2 h_K \|u_h\|_K.
\end{equation}

By definition of $\dgonorm{\cdot}$, $\dgonorm{u_h}^2 \ge \omega^2\|u_h\|_K^2$, and by \eqref{eq:AK0_scaled_coerc},
\[
\|A_{K,0}u_h\|_{\QKh'} \ge c_{A,0} \dgonorm{u_h}.
\]
Using these in \eqref{eq:perturb_loccoerc} gives
\[
\|(\AK-A_{K,0})u_h\|_{\QKh'}
\le \frac{2\sqrt{\const{pf}}}{c_{A,0}}\omega h_K \|A_{K,0}u_h\|_{\QKh'}
= \gamma_K\|A_{K,0}u_h\|_{\QKh'}.
\]
i.e.\ \eqref{eq:nearlycontant} holds with $\gamma=\sup_K\gamma_K$.
For $\gamma<1$ \cref{lem:abstractneumann} gives the result.
\end{proof}

\begin{lemma}[Continuity of $\AK$]\label{lem:AK_cont}
Assume that (A0) holds.
Then for all $u\in \Vsh\cap H^2(\Th)$ there holds
\begin{equation}\label{eq:AK_cont_helmholtz}
\sum_{K\in\Th}\|\AK u\|_{\QKh'}^2 \le \CA^2\,\tdgonorm{u}^2,
\end{equation}
with the fully explicit choice
\[
\CA^2 = 4 + 2\,\const{pf}C_\Omega^2.
\]
\end{lemma}
\begin{proof}
Fix $K\in\Th$ and $u\in H^2(K)$. By definition and the triangle inequality,
\[
\|\AK u\|_{\QKh'}
\le \|A_{K,0}u\|_{\QKh'} + \omega^2 h_K \sup_{q_h\neq 0}\frac{|\langle u,q_h\rangle_K|}{\|q_h\|_{\QKh}}.
\]
The prototype continuity estimate (cf.\ \cref{lem:nonharmonicNE}, continuity part) gives
\[
\sum_{K\in\Th}\|A_{K,0}u\|_{\QKh'}^2 \le 2\,\tdgonorm{u}^2.
\]
For the mass term, Cauchy--Schwarz and \eqref{pf} yield for any $q_h\in\QKh$
\[
\omega^2 h_K|\langle u,q_h\rangle_K| \le \omega^2 h_K\|u\|_{K}\,\|q_h\|_{K}
\le 
\omega^2 h_K\|u\|_{K}\sqrt{\const{pf}}\norm{q_h}_{\QKh},
\]
Using \eqref{ass:oh} we have $\omega^2 h \le C_\Omega\frac{\omega^2}{1+\omega^2}\le C_\Omega\,\omega, $ and therefore
\[
\|\AK u\|_{\QKh'}
\le \|A_{K,0}u\|_{\QKh'} + \sqrt{\const{pf}}C_\Omega\,\omega \|u\|_{K}.
\]
Using $(a+b)^2\le 2(a^2+b^2)$ and summing over $K$ yields
\[
\sum_{K\in\Th}\|\AK u\|_{\QKh'}^2
\le 2\sum_{K\in\Th}\left(\|A_{K,0}u\|_{\QKh'}^2
+ \const{pf}C_\Omega^2\omega^2 \|u\|_{K}^2\right).
\]
Combining the bounds and using $\omega^2\|u\|_{\Th}^2\le \tdgonorm{u}^2$ gives \eqref{eq:AK_cont_helmholtz} with
$\CA^2 = 4 + 2\,\const{pf}C_\Omega^2$.
\end{proof}

\subsection{Analysis of the global problem}\label{sec:global_analysis}
To apply \cref{cor:cea} we will now verify $T$-coercivity of the global problem on the Trefftz space \eqref{eq:ah_coerc} in \cref{sec:ah_coerc}.
And continuity \eqref{eq:ah_cont} in \cref{sec:ah_cont}.

\subsubsection{Continuity of \texorpdfstring{$a_h(\cdot,\cdot)$}{ah}}\label{sec:ah_cont}

\begin{lemma}\label{lem:ah_continuity}
Let $(u,v_h)\in V_{*h}\times V_h$. Then
\[
|a_h(u,v_h)|\ \le\ \Ca\,\tdgonorm{u}\,\dgonorm{v_h},
\]
with
\[
\Ca = 3+\alpha+\const{itr}\,C_\sharp^{1/2},
\]
where $C_\sharp$ is the bounded number of facets per element.
\end{lemma}

\begin{proof}
By Cauchy--Schwarz on the volume and penalty terms,
\begin{align*}
|a_h(u,v_h)|
\le&
\|\nabla u\|_\Th\|\nabla v_h\|_\Th
+\omega^2\|u\|_\Th\|v_h\|_\Th
+\alpha\frac{p^2}{h}\|\jmp{u}\|_{\Fhi}\|\jmp{v_h}\|_{\Fhi}
+\omega\|u\|_{\Fhb}\|v_h\|_{\Fhb}
\\
&\quad
+\|\avg{\nabla u\cdot n}\|_{\Fhi}\|\jmp{v_h}\|_{\Fhi}
+\|\jmp{u}\|_{\Fhi}\|\avg{\nabla v_h\cdot n}\|_{\Fhi}.
\end{align*}
Using $\|\avg{\phi}\|_{\Fhi}^2\le \frac12\sum_{K}\|\phi\|_{\partial K\cap\Fhi}^2$ and the definitions of
$\tdgonorm{\cdot}$, $\dgonorm{\cdot}$,
\[
\|\avg{\nabla u\cdot n}\|_{\Fhi}\,\|\jmp{v_h}\|_{\Fhi}
\le \frac1{\sqrt2}\tdgonorm{u}\,\dgonorm{v_h}.
\]
For the second coupling term, the inverse trace \eqref{itr} applied to $r_h=\nabla v_h\cdot n$ on each facet yield
\[
\|\jmp{u}\|_{\Fhi}\,\|\avg{\nabla v_h\cdot n}\|_{\Fhi}
\le \const{itr}\,C_\sharp^{1/2}\,\tdgonorm{u}\,\dgonorm{v_h}.
\]
Collecting the bounds gives the result with $\Ca=3+\alpha+\const{itr}C_\sharp^{1/2}$.
\end{proof}

\subsubsection{\texorpdfstring{$T$-coercivity of $a_h(\cdot,\cdot)$}{T-coercivity of ah}}\label{sec:ah_coerc}
We follow a Schatz argument, using the regularity of the Helmholtz problem and the approximation properties of the Trefftz space to construct a suitable test function for $T$-coercivity.
Before we can do this, we need to verify the relation between the $\tdgonorm{\cdot}$ and $\dgonorm{\cdot}$ norms on the discrete space $V_h$, verifying assumption \eqref{eq:Cstar}.

\begin{lemma}
For polynomials $u_h\in V_h$ there holds \eqref{eq:Cstar}, i.e.
\begin{equation}\label{cd}
    \tdgonorm{u_h} \leq \CD \dgonorm{u_h},
\end{equation}
with $\CD = \sqrt{1 + \const{itr}^2}$.
\end{lemma}
\begin{proof}
Using the inverse trace inequality \eqref{itr} on each element gives
\[
p^{-2} h_K \norm{\grad u_h \cdot n}_{\partial K\cap\Fhi}^2
\le \const{itr}^2 \norm{\grad u_h}_{K}^2 .
\]
Summing over $K\in\Th$, we obtain
\[
\tdgonorm{u_h}^2
\le
\dgonorm{u_h}^2 + \const{itr}^2 \norm{\grad u_h}_{\Th}^2
\le
(1+\const{itr}^2)\dgonorm{u_h}^2 .
\]
Taking square roots yields \eqref{cd}.
\end{proof}

We will make use of the following regularity result for the Helmholtz problem. 

\begin{lemma}[Helmholtz regularity {\cite[Prop.\ 8.1.4]{MM95}}]\label{lem:regularity}
Let $\Omega$ be a bounded star-shaped domain with smooth boundary or a bounded convex domain.
Then there exists $C_\dom>0$, depending only on $\Omega$, such that for any
$f\in L^2(\Omega)$ and $g\in H^{1/2}(\partial\Omega)$ the solution $u$ of \eqref{eq:helmholtz} satisfies
\begin{subequations}\label{eq:regularity}
\begin{align}
    \norm{u}_{V,\omega} &\leq C_\dom \bigl(\norm{f}_\Omega + \norm{g}_{\partial\Omega}\bigr),\\
    \snorm{u}_{H^2(\Omega)} &\leq C_\dom \Bigl[(1+|\omega|)\bigl(\norm{f}_\Omega + \norm{g}_{\partial\Omega}\bigr)
    +\norm{g}_{H^{1/2}(\partial\Omega)}\Bigr].
\end{align}
\end{subequations}
\end{lemma}

\begin{lemma}[Trefftz approximation of $H^2$-functions]\label{lem:h2approx}
Assume \textup{(A0)}--\textup{(A3)} and
\[
1-\frac{2\sqrt{\const{pf}}}{c_{A,0}}\sqrt{C_\Omega h}\ge C_2>0
\]
for a constant $C_2$ independent of $h$ and $\omega$.
Then for every $z\in H^2(\Omega)$ there exists $z_h\in\ITh$ such that
\begin{align}
\tdgonorm{z-z_h}
&\le \const{h2appra}^{(2)} h \snorm{z}_{H^2(\Omega)}
 + \const{h2appra}^{(V)} \omega h \norm{z}_{V,\omega}.
\label{h2appra}
\end{align}
with
\[
\const{h2appra}^{(2)}
:= C_{\rm conf}\Bigl(1+2\CD \sqrt{\const{pf}}\,c_{A,0}^{-1}C_2^{-1} C_\Omega^2\Bigr),
\qquad
\const{h2appra}^{(V)}
:= 2\CD \sqrt{\const{pf}}\,c_{A,0}^{-1}C_2^{-1}.
\]
\end{lemma}
\begin{proof}
By standard conforming approximation theory, for any $z\in H^2(\Omega)$ there exists
$z_h^{c}\in\IP^1_c(\Th)$ such that
\begin{align}
\norm{z-z_h^{c}}_{\Th} &\le C_\dom h^2 \snorm{z}_{H^2(\Omega)}, \label{eq:h2approx:L2}\\
\|\nabla(z-z_h^{c})\|_{\Th} &\le C_\dom h \snorm{z}_{H^2(\Omega)}. \label{eq:h2approx:H1}
\end{align}
Since $z-z_h^{c}\in H^1(\Omega)$ is conforming, it has no jump terms. Using
\eqref{eq:h2approx:L2}--\eqref{eq:h2approx:H1}, the trace inequality \eqref{tr}, and \eqref{ass:oh},
there exists a constant $C_{\rm conf}>0$, depending only on $C_\dom$, $C_\Omega$,
$\const{tr}$, and the mesh-shape constants, such that
\begin{equation}\label{eq:tdgo_conforming_est}
\tdgonorm{z-z_h^{c}}\ \le\ C_{\rm conf}\, h \snorm{z}_{H^2(\Omega)}.
\end{equation}

Let $w_h\in \Lh$ be defined elementwise by
\[
A_K w_h = A_K z_h^{c} \quad\text{in }\QKh' \qquad \forall K\in\Th.
\]
Since $z_h^{c}|_K\in \IP^1(K)$ is harmonic, $-\Delta z_h^{c}=0$ on each element, hence
\[
A_K z_h^{c} = -\omega^2 h_K z_h^{c}.
\]
Using the local coercivity \eqref{eq:loccoerc} elementwise and summing over $K\in\Th$, we obtain
\begin{equation}\label{eq:wh_dgo}
\dgonorm{w_h}\ \le\ \cA^{-1} \|\ATh z_h^{c}\|_{\Qh'}
= \cA^{-1} \omega^2 \|h_K z_h^{c}\|_{\Qh'}.
\end{equation}
Moreover, by definition of the dual norm and Cauchy--Schwarz,
\[
\|h_K z_h^{c}\|_{\QKh'}
\le h_K\|z_h^{c}\|_{K}\sup_{q_h\neq0}\frac{\|q_h\|_K}{\|q_h\|_{\QKh}}
\le 2\sqrt{\const{pf}}\, h_K\|z_h^{c}\|_{K}.
\]
Summing over $K$ yields
\begin{equation}\label{eq:hKzh_Qdual_global}
\|h_K z_h^{c}\|_{\Qh'} \le 2\sqrt{\const{pf}}\, h \|z_h^{c}\|_{\Th}.
\end{equation}
Combining \eqref{eq:wh_dgo} and \eqref{eq:hKzh_Qdual_global} gives
\begin{equation}\label{eq:whzh}
\dgonorm{w_h}\ \le\ 2\sqrt{\const{pf}}\,\cA^{-1}\omega^2 h \|z_h^{c}\|_{\Th}.
\end{equation}

Set $\tilde z_h:=z_h^{c}-w_h$. Then $\tilde z_h\in\ITh$ by construction and, using \eqref{cd},
\begin{align*}
\tdgonorm{z-\tilde z_h}
&\le \tdgonorm{z-z_h^{c}} + \tdgonorm{w_h}
\le \tdgonorm{z-z_h^{c}} + \CD \dgonorm{w_h}
\\
&\le C_{\rm conf} h \snorm{z}_{H^2(\Omega)}
+ 2\CD \sqrt{\const{pf}}\,\cA^{-1} \omega^2 h \|z_h^{c}\|_{\Th}.
\end{align*}
By the triangle inequality and \eqref{eq:h2approx:L2},
\[
\|z_h^{c}\|_{\Th}\le \|z\|_{\Th}+\|z-z_h^{c}\|_{\Th}
\le \omega^{-1}\|z\|_{V,\omega} + C_\dom h^2 \snorm{z}_{H^2(\Omega)},
\]
since $\|z\|_{V,\omega}^2=\|\nabla z\|_\Omega^2+\omega^2\|z\|_\Omega^2\ge \omega^2\|z\|_\Omega^2$.
Inserting this and, if necessary, enlarging $C_{\rm conf}$ so that
$C_{\rm conf}\ge C_\dom$, we obtain
\begin{align*}
\tdgonorm{z-\tilde z_h}
&\le C_{\rm conf} h \snorm{z}_{H^2(\Omega)}
+ 2\CD \sqrt{\const{pf}}\,\cA^{-1} \omega^2 h
\Bigl(\omega^{-1}\|z\|_{V,\omega} + C_\dom h^2 \snorm{z}_{H^2(\Omega)}\Bigr)
\\
&\le \Bigl(1+2\CD \sqrt{\const{pf}}\,\cA^{-1} \omega^2 h^2\Bigr)C_{\rm conf} h \snorm{z}_{H^2(\Omega)}
+ 2\CD \sqrt{\const{pf}}\,\cA^{-1} \omega h \|z\|_{V,\omega}.
\end{align*}

Finally we make $\cA^{-1}$ explicit. By \cref{lem:loccoerc},
$
\cA=c_{A,0}(1-\gamma),
\text{ and }
\gamma\le \frac{2\sqrt{\const{pf}}}{c_{A,0}} \omega h.
$
Using \eqref{ass:oh} gives $\omega h\le \sqrt{C_\Omega h}$ and hence
\[
\gamma\le \frac{2\sqrt{\const{pf}}}{c_{A,0}} \sqrt{C_\Omega h}.
\]
By the assumption
$
1-\frac{2\sqrt{\const{pf}}}{c_{A,0}}\sqrt{C_\Omega h}\ge C_2
$
we have $1-\gamma\ge C_2$, and therefore
\begin{equation}\label{eq:cAinv}
\cA^{-1}\le c_{A,0}^{-1}C_2^{-1}.
\end{equation}
Moreover, \eqref{ass:oh} implies $\omega^2 h^2\le C_\Omega^2$.
Inserting these bounds yields \eqref{h2appra} with the stated constants
(and with $z_h:=\tilde z_h$).
\end{proof}

\begin{lemma}[ $T$-coercivity on $\ITh$]\label{lem:Tcoercive}
Assume the assumptions of Lemma~\ref{lem:h2approx}.
Assume moreover $\alpha>\const{itr}^2$ and set
\[
C_\alpha:=\frac{\alpha-\const{itr}^2}{1+\alpha}>0.
\]
Let $\theta:=-\pi/4$ and define
$
\delta_h:=2\sqrt2 \Ca C_\dom (\const{h2appra}^{(2)}+\const{h2appra}^{(V)}) (\omega h) (1+|\omega|).
$
Assume in addition that $\delta_h<C_\alpha$.
Then there exists a linear injective operator $T_h:\ITh\to V_h$ with
$\dgonorm{T_hu_h}\le \dgonorm{u_h}$ such that
\[
\Re\bigl(e^{-i\theta}a_h(u_h,T_h u_h)\bigr)\ \ge\ \ca \dgonorm{u_h}^2,
\qquad
\ca:=\frac{C_\alpha-\delta_h}{\sqrt2(1+C_T)},
\]
where
$
C_T:=2C_\dom\Bigl(|\omega|+(\const{h2appra}^{(2)}+\const{h2appra}^{(V)}) C_\Omega(1+|\omega|)\Bigr).
$
\end{lemma}
\begin{proof}
Fix $\theta=-\pi/4$, so that for any $v$
\begin{equation}\label{eq:phase_id}
\sqrt2 \Re\bigl(e^{-i\theta}a_h(u_h,v)\bigr)=\Re a_h(u_h,v)+\Im a_h(u_h,v).
\end{equation}
Let $u_h\in\ITh$ be given and let $z\in H^1(\Omega)$ be arbitrary. Using Cauchy--Schwarz,
the inverse trace inequality \eqref{itr} applied to $\nabla u_h\cdot n$, and the elementary inequality
$a^2-2\gamma ab+\alpha b^2\ge\frac{\alpha-\gamma^2}{1+\alpha}(a^2+b^2)$ with $\gamma=\const{itr}$, we obtain
\begin{align*}
\Re a_h(u_h,u_h+z)
&\ge
C_\alpha\Bigl(\|\nabla u_h\|_\Th^2+\frac{p^2}{h}\|\jmp{u_h}\|_{\Fhi}^2\Bigr)
-\omega^2\|u_h\|_\Omega^2
+\Re a_h(u_h,z),
\\
\Im a_h(u_h,u_h+z)&=\omega\|u_h\|_{\Fhb}^2+\Im a_h(u_h,z).
\end{align*}
Choose $z\in H^1(\Omega)$ as the solution of the adjoint problem
\begin{equation}\label{eq:z_adj}
a(v,z)=2\omega^2 (v,u_h)_\Omega \qquad \forall v\in H^1(\Omega).
\end{equation}
By adjoint consistency, $a_h(v_h,z)=2\omega^2(v_h,u_h)_\Omega$ for all $v_h\in\ITh$, hence
$a_h(u_h,z)=2\omega^2\|u_h\|_\Omega^2\in\mathbb R$. Therefore,
\begin{align}
(\Re+\Im)a_h(u_h,u_h+z)
&\ge
C_\alpha\Bigl(\|\nabla u_h\|_\Th^2+\frac{p^2}{h}\|\jmp{u_h}\|_{\Fhi}^2\Bigr)
+\omega^2\|u_h\|_\Omega^2
+\omega\|u_h\|_{\Fhb}^2
\nonumber\\
&\ge C_\alpha \dgonorm{u_h}^2. \label{eq:baseline}
\end{align}

Let $z_h\in\ITh$ be the approximant produced by Lemma~\ref{lem:h2approx} (with the fixed constructive choice,
so that $z\mapsto z_h$ is linear), and define
\[
\widetilde T_hu_h:=u_h+z_h \in V_h .
\]
Using \eqref{eq:phase_id}, $|\Re w+\Im w|\le \sqrt2 |w|$, the continuity of $a_h$, and \eqref{eq:baseline},
\begin{align*}
\sqrt2 \Re\bigl(e^{-i\theta}a_h(u_h,\widetilde T_hu_h)\bigr)
&=(\Re+\Im)a_h(u_h,u_h+z_h)
\\
&\ge (\Re+\Im)a_h(u_h,u_h+z)-|(\Re+\Im)a_h(u_h,z-z_h)|
\\
&\ge C_\alpha \dgonorm{u_h}^2-\sqrt2 |a_h(u_h,z-z_h)|
\\
&\ge C_\alpha \dgonorm{u_h}^2-\sqrt2 \Ca \dgonorm{u_h} \tdgonorm{z-z_h}.
\end{align*}
Lemma~\ref{lem:h2approx} gives
\[
\tdgonorm{z-z_h}
\le \const{h2appra}^{(2)} h \snorm{z}_{H^2(\Omega)}
     +\const{h2appra}^{(V)} \omega h \|z\|_{V}.
\]
For $z$ solving \eqref{eq:z_adj}, Lemma~\ref{lem:regularity} (with $f=2\omega^2u_h$, $g=0$) yields
\[
\|z\|_{V,\omega}\le 2C_\dom \omega^2\|u_h\|_\Omega,
\qquad
\snorm{z}_{H^2(\Omega)}\le 2C_\dom(1+|\omega|) \omega^2\|u_h\|_\Omega.
\]
Using $\|z\|_{V}\le \|z\|_{V,\omega}$ and $\omega\|u_h\|_\Omega\le \dgonorm{u_h}$, we obtain
\[
\tdgonorm{z-z_h}
\le 2C_\dom (\const{h2appra}^{(2)}+\const{h2appra}^{(V)})(\omega h)(1+|\omega|) \dgonorm{u_h}.
\]

We now take care of the normalization of $\widetilde T_h$.
Using triangle inequality 
\[
\dgonorm{z_h}\le \dgonorm{z}+\tdgonorm{z-z_h}.
\]
Moreover, $\|z\|_{V,\omega}\le 2C_\dom \omega^2\|u_h\|_\Omega\le 2C_\dom|\omega| \dgonorm{u_h}$,
and the estimate for $\tdgonorm{z-z_h}$ above gives
\[
\dgonorm{z_h}\le 2C_\dom\Bigl(|\omega|+(\const{h2appra}^{(2)}+\const{h2appra}^{(V)})(\omega h)(1+|\omega|)\Bigr)\dgonorm{u_h}.
\]
Using \eqref{ass:oh}, i.e. $(1+\omega^2)h\le C_\Omega$, implies $\omega h\le C_\Omega$, hence
\[
\dgonorm{\widetilde T_hu_h}\le(1+C_T)\dgonorm{u_h}.
\]
Finally set $T_h:=(1+C_T)^{-1}\widetilde T_h$. Then $\dgonorm{T_hu_h}\le\dgonorm{u_h}$ and
\[
\Re \bigl(e^{-i\theta}a_h(u_h,T_hu_h)\bigr)
=\frac{1}{1+C_T}\Re \bigl(e^{-i\theta}a_h(u_h,\widetilde T_hu_h)\bigr)
\ge \frac{C_\alpha-\delta_h}{\sqrt2(1+C_T)}\dgonorm{u_h}^2,
\]
which is the claimed $T$-coercivity.
\end{proof}

\subsection{Quasi-best approximation}\label{sec:quasi_best}

In the previous subsection we have verified the assumptions of \cref{cor:cea}. This allows us to formulate the following quasi-best approximation result. 

\begin{theorem}\label{thm:best_approx}
Assume \textup{(A0)}--\textup{(A3)}.
Let $\CD,\,\cA,\,\CA,\,\ca,\,\Ca$ be as in \cref{lem:loccoerc,lem:AK_cont,lem:ah_continuity,lem:Tcoercive}, and let $u_h\in V_h$ solve \eqref{eq:PDEh} for the Helmholtz problem.
Assume $u\in H^2(\Th)$ so that $\tdgonorm{u-v_h}$ is finite.
Then
\begin{equation}\label{eq:qopt}
\dgonorm{u-u_h}
\le
(1+C_{\cea}^{(1)}) 
\inf_{v_h\in V_h}\tdgonorm{u-v_h},
\end{equation}
with $C_{\cea}^{(1)}$ as in \cref{cor:cea}.
\medskip
Assume in addition that $u\in H^{s+1}(\Omega)$ for some integer $1\le s\le p$.
Assume that $\Th$ admits a simplex covering $\Th^\sharp$ satisfying Definition~4.27 and Assumption~4.28 of \cite{CDG22}, and let $E:H^{s+1}(\Omega)\to H^{s+1}(\IR^d)$ be the extension operator from Theorem~4.30 therein with stability constant $C_E$.
Then there exists $C_{\rm app}>0$, independent of $h$, $\omega$ and $p$, such that
\begin{eqs}\label{eq:hp_final}
&\dgonorm{u-u_h}\\
&\qquad\le
(1+C_{\cea}^{(1)}) C_{\rm app}^{1/2}
\Biggl(
\sum_{K\in\Th}
\Bigl(\frac{h_K}{p}\Bigr)^{2s}
\Bigl[
(1+p)
+\omega^2\Bigl(\frac{h_K}{p}\Bigr)^{2}
+\omega \Bigl(\frac{h_K}{p}\Bigr)
\Bigr] 
|u|_{H^{s+1}(K)}^2
\Biggr)^{1/2}.
\end{eqs}
The constant $C_{\rm app}$ depends only on $d$, the shape-regularity of the covering simplices, the covering
constants $O_\Omega$ and $C_{\rm diam}$, the extension constant $C_E$, the approximation constants in
Lemma~4.31 of \cite{CDG22}, and the mesh-shape constants entering \eqref{tr} and the face-size comparability in
\textup{(A2)}.
\end{theorem}
\begin{proof}
By consistency the residual terms in \cref{cor:cea} vanish, hence for any $v_h\in V_h$,
\[
    \dgonorm{u_h-v_h}\le C_{\cea}^{(1)}\,\tdgonorm{u-v_h}.
\]
Using $u-u_h=(u-v_h)+(v_h-u_h)$, the triangle inequality, and $\dgonorm{\cdot}\le \tdgonorm{\cdot}$ gives
\[
    \dgonorm{u-u_h}\le (1+C_{\cea}^{(1)})\,\tdgonorm{u-v_h},
\]
and taking the infimum yields \eqref{eq:qopt}. The quasi-best approximation \eqref{eq:hp_final} follows from \eqref{eq:qopt} and standard $hp$-approximation theory for polynomials on shape-regular simplices.
\end{proof}

\begin{theorem}[$L^2$-error bound for embedded Trefftz--DG Helmholtz]\label{th:L2_helmholtz}
Let the assumptions of \cref{lem:h2approx} hold, and assume that the Helmholtz solution $u$ is sufficiently regular
so that $\tdgonorm{u-v_h}$ is finite for the approximants $v_h\in V_h$.

Assume moreover that $\Omega$ is either convex or star-shaped with smooth boundary so that the (continuous) dual problem
associated with the right-hand side $\phi\in L^2(\Omega)$ admits a solution $z\in H^2(\Omega)$ satisfying the bounds
\begin{equation}\label{eq:dual_reg_exp}
\|z\|_{V,\omega} \le C_{\rm dual}\,\|\phi\|_{L^2(\Omega)},
\qquad
|z|_{H^2(\Omega)} \le C_{\rm dual}\,(1+\omega)\,\|\phi\|_{L^2(\Omega)},
\end{equation}
for a constant $C_{\rm dual}>0$ independent of $h$ and $\omega$, see \cite[Proposition 8.1.4]{MM95}

Let $u_h\in V_h$ be the embedded Trefftz--DG solution \eqref{eq:PDEh} of the Helmholtz problem.
Then
\begin{equation}\label{eq:L2_helmholtz_bound}
\|u-u_h\|_{L^2(\Omega)}
\ \le\
C_{L^2}\,h(1+\omega)\,\inf_{v_h\in V_h}\tdgonorm{u-v_h},
\end{equation}
with the explicit constant
\begin{equation}\label{eq:CL2_def}
C_{L^2}
:=
\Ca\;\Bigl(C_{\rm reg}^{(1)}+2C_{\rm reg}^{(2)}\Bigr)\;
\Bigl(1+\CD+\CD\,C_{\rm qopt}\Bigr).
\end{equation}
where $\CD$ is the constant from \eqref{cd},
and $ C_{\rm qopt}:=1+C_{\cea}^{(1)} $ is the quasi-best approximation constant from \cref{thm:best_approx}.
\begin{equation}\label{eq:Creg_choice}
C_{\rm reg}^{(1)}=C_{\rm reg}^{(2)}
:= C_{\rm dual}\,\bigl(\const{h2appra}^{(2)}+\const{h2appra}^{(V)}\bigr),
\qquad
h_H:=h(1+\omega),
\end{equation}
with $\const{h2appra}^{(2)},\const{h2appra}^{(V)}$ as in \cref{lem:h2approx}.
\end{theorem}

\begin{proof}
Set $\IH:=L^2(\Omega)$ and apply \cref{th:aubinnitsche}.
Following the same arguments as in \cref{lem:ah_continuity} implies \eqref{eq:a-bit-cont-exp} with $C_{\rm adj}=\Ca$.

Let $e:=u-u_h$ and let $z$ be the dual solution with right-hand side $\phi=e$, so that
$\sup_{v\in V}\frac{|a(v,z)|}{\|v\|_{L^2(\Omega)}}=\|e\|_{L^2(\Omega)}$.
By \cref{lem:h2approx}, there exists $z_h\in \ITh\subseteq\tch\ITh$ such that
\[
\tdgonorm{z-z_h}
\le \const{h2appra}^{(2)}\,h\,|z|_{H^2(\Omega)}
    +\const{h2appra}^{(V)}\,\omega h\,\|z\|_{V,\omega}.
\]
Using \eqref{eq:dual_reg_exp} with $\phi=e$ and $\omega\le 1+\omega$ yields
\[
\tdgonorm{z-z_h}
\le
C_{\rm dual}\,h(1+\omega)\,\bigl(\const{h2appra}^{(2)}+\const{h2appra}^{(V)}\bigr)\,\|e\|_{L^2(\Omega)}.
\]
Since $\dgonorm{\cdot}\le \tdgonorm{\cdot}$, the same bound holds with $\dgonorm{\cdot}$ in place of $\tdgonorm{\cdot}$.
Thus \eqref{eq:H-reg-bound-exp}--\eqref{eq:H-reg-bound2-exp} hold with the choices \eqref{eq:Creg_choice}.

Finally, the Helmholtz method is consistent, hence the residual terms in \cref{th:aubinnitsche} vanish.
Moreover, \cref{thm:best_approx} gives
\[
\dgnorm{u-u_h}\le C_{\rm qopt}\inf_{v_h\in V_h}\tdgonorm{u-v_h}.
\]
Combining these two facts with \cref{th:aubinnitsche} (in the residual-free case) yields
\eqref{eq:L2_helmholtz_bound} with the constant \eqref{eq:CL2_def}.
\end{proof}
From this result, we see that the $L^2$-error is of optimal order, as in \eqref{eq:hp_final}, with the additional factor of $h(1+\omega)$.

\section{Numerical Experiments}\label{sec:numerics}

For the implementation of the methods we use \texttt{NGSolve} \cite{ngsolve} and \texttt{NGSTrefftz} \cite{ngstrefftz}. Code is available at \cite{stocker_2026_18985575}.

In \cref{sec:hconvergence} we confirm the theoretical convergence rates for $h$--refinement. 
We consider $p$--refinement in \cref{sec:pconvergence} and observe exponential convergence in the preasymptotic regime.
In the last two subsections \cref{sec:largeo,sec:varo} we study performance outside the scope of the theory, particularly the effect of $p$--refinement on the preasymptotic regime and smoothly varying wave numbers.

Throughout, we use simplicial quasi-uniform shape-regular meshes and solve all linear systems by a direct solver.
The penalty parameter is fixed to $\alpha=10$ in all experiments.
For each test, we compare the embedded Trefftz DG method and the standard DG method on the same sequence of meshes and with the same polynomial degree $p$, so that the only difference is the choice of local discrete space.
When plotting the error against \texttt{ndof}, this refers to the total number of global degrees of freedom of the corresponding discrete problem.

\subsection{\texorpdfstring{$h$}{h}-convergence}\label{sec:hconvergence}

In this section, we consider a known analytical solution on the unit square $\Omega = (0,1)^2$. 
We choose the impedance boundary condition $g$ such that the exact solution is given by
\begin{equation}\label{eq:hankel}
u(x,y) = \calH_0^{(1)}\bigl(\omega \,\| (x - x_0, y) \|\bigr),
\end{equation}
where $\calH_0^{(1)}$ denotes the Hankel function of the first kind of order zero and $x_0 = (-0.25,0)$, and we take the wave number $\omega = 10$.
For this solution, the right-hand side $f$ vanishes.

This example provides a convenient benchmark for $h$--convergence, since the exact solution is smooth in $\Omega$ while still exhibiting a genuinely oscillatory Helmholtz behavior.
Moreover, the singularity of the Hankel fundamental solution is located outside the computational domain, so that the observed convergence is not polluted by reduced interior regularity.
On each simplicial quasi-uniform mesh we compare both methods for the same polynomial degree $p$ and the same stabilization parameter.

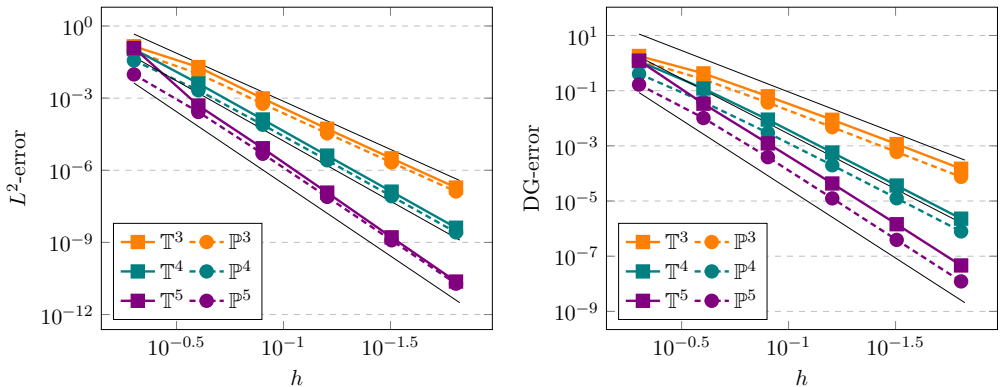
\begin{figure}[ht!]\centering
\resizebox{.9\linewidth}{!}{
\begin{tikzpicture}
	\begin{groupplot}[
	group style={
	group name={my plots},
	group size=2 by 1,
	horizontal sep=2cm,
	},
	ymajorgrids=true,
	grid style=dashed,
	]      
    \nextgroupplot[ymode=log,xmode=log,x dir=reverse, ylabel={$L^2$-error},xlabel={$h$}, cycle list name=paulcolors2, legend columns=2, legend pos=south west]
    \foreach \k in {3,4,5}{
    \addplot+[discard if not={p}{\k},discard if not={method}{etvol}] table [x=h, y=l2error, col sep=comma] {./numex/hankel2d.csv};
    \addplot+[discard if not={p}{\k},discard if not={method}{dgvol}] table [x=h, y=l2error, col sep=comma] {./numex/hankel2d.csv};
    }
    \addlegendimage{solid}
    \addplot[domain=0.015:.5] {exp(-4*ln(1/x)+2.0)};
    \addplot[domain=0.015:.5] {exp(-5*ln(1/x)+0.5)};
    \addplot[domain=0.015:.5] {exp(-6*ln(1/x)-1.3)};
    \legend{$\IT^3$,$\IP^3$,$\IT^4$,$\IP^4$,$\IT^5$,$\IP^5$}%

    \nextgroupplot[ymode=log,xmode=log,x dir=reverse, ylabel={DG-error},xlabel={$h$}, cycle list name=paulcolors2, legend columns=2, legend pos=south west]
    \foreach \k in {3,4,5}{
    \addplot+[discard if not={p}{\k},discard if not={method}{etvol}] table [x=h, y=dgerror, col sep=comma] {./numex/hankel2d.csv};
    \addplot+[discard if not={p}{\k},discard if not={method}{dgvol}] table [x=h, y=dgerror, col sep=comma] {./numex/hankel2d.csv};
    }
    \addlegendimage{solid}
    \addplot[domain=0.015:0.5] {exp(-3*ln(1/x)+4.5)};
    \addplot[domain=0.015:0.5] {exp(-4*ln(1/x)+3.3)};
    \addplot[domain=0.015:0.5] {exp(-5*ln(1/x)+1.0)};
    \legend{$\IT^3$,$\IP^3$,$\IT^4$,$\IP^4$,$\IT^5$,$\IP^5$}%
	\end{groupplot}
\end{tikzpicture}}
\vspace{-0.5em}
\caption{
    Convergence of the embedded Trefftz DG method and the standard DG method for the problem with exact solution \eqref{eq:hankel}.
    Theoretical rates are indicated by the lines, i.e., $\mathcal O(h^{p+1})$ for the $L^2$-error and $\mathcal O(h^{p})$ for the DG-error.
}
\label{fig:hankel}
\end{figure}

In \cref{fig:hankel} we compute the solution for different mesh sizes $h$ and polynomial degrees $p=3,4,5$ and compare the embedded Trefftz DG method with the standard DG method using full polynomial spaces.
We observe in both the $L^2$-error and the DG-error the expected convergence rates of order $\mathcal O(h^{p+1})$ and $\mathcal O(h^{p})$, respectively, for both methods.
Thus, Trefftz constraint retains the asymptotic convergence behavior, while reducing the local approximation space.

\subsection{\texorpdfstring{$p$}{p}-convergence}\label{sec:pconvergence}
We consider the domain $\Omega = (0,1)^2$ with right-hand side and impedance boundary condition such that the exact solution is given by
\begin{equation}\label{eq:sinsin}
u(x,y) = \sin(\pi x)\sin(\pi y),
\end{equation}
setting the wave number to $\omega = 1$.
For this experiment we keep the mesh fixed and increase the polynomial degree $p$, so that the figure isolates the effect of local approximation order.
Since the exact solution is smooth, one expects exponential convergence in the pre-asymptotic regime under $p$--refinement.

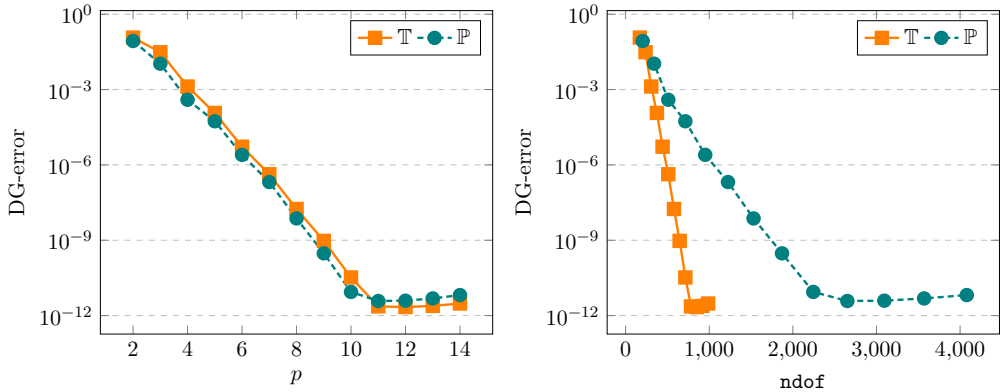
\begin{figure}[ht!]\centering
\resizebox{.9\linewidth}{!}{
\begin{tikzpicture}
	\begin{groupplot}[
	group style={
	group name={my plots},
	group size=2 by 1,
	horizontal sep=2cm,
	},
	ymajorgrids=true,
	grid style=dashed,
	]      
    \nextgroupplot[ymode=log,ylabel={DG-error},xlabel={$p$}, legend columns=3, legend pos=north east, cycle list name=paulcolors1]
    \addplot+[discard if not={method}{etvol},discard if not={hnr}{0}] table [x=p, y=dgerror, col sep=comma] {./numex/sinsin2d.csv};
    \addplot+[discard if not={method}{dgvol},discard if not={hnr}{0}] table [x=p, y=dgerror, col sep=comma] {./numex/sinsin2d.csv};
    \legend{$\IT$,$\IP$}

    \nextgroupplot[ymode=log,ylabel={DG-error},xlabel={\texttt{ndof}}, legend columns=3, legend pos=north east, cycle list name=paulcolors1]
    \addplot+[discard if not={method}{etvol},discard if not={hnr}{0}] table [x=dofs, y=dgerror, col sep=comma] {./numex/sinsin2d.csv};
    \addplot+[discard if not={method}{dgvol},discard if not={hnr}{0}] table [x=dofs, y=dgerror, col sep=comma] {./numex/sinsin2d.csv};
    \legend{$\IT$,$\IP$}

	\end{groupplot}
\end{tikzpicture}}
\vspace{-0.5em}
\caption{
    Results for the problem with exact solution \eqref{eq:sinsin} for the embedded Trefftz DG method and the standard DG method.
    Left: DG-error against polynomial degree $p$. Right: DG-error against number of degrees of freedom.
}
\label{fig:sinsin}
\end{figure}

This is confirmed in \cref{fig:sinsin} for both discretizations, were we consider $p=2,\ldots,14$.
When the error is plotted against the total number of degrees of freedom, the advantage of the embedded Trefftz space becomes more visible: for a given accuracy, the embedded Trefftz discretization requires substantially fewer unknowns than the standard DG method.
Although the lower bound in \cref{lem:nonharmonicNE} exhibits a problematic dependence on $p$, this behaviour is not reflected in the present experiment: both methods show exponential convergence in $p$, as expected for a smooth solution.

\subsection{Large wave numbers} \label{sec:largeo}
We consider the problem \eqref{eq:helmholtz} with $f=0$ and $g$ chosen such that the exact solution is 
\begin{equation}\label{eq:exp}
    u(x,y) = \exp(i\omega (x-y)/\sqrt{2}).
\end{equation}
The domain is the unit disk $\Omega = B_1(0)$.
This experiment follows the setup chosen in \cite{BM19} and is meant to investigate the behavior of the pre-asymptotic region for different polynomial degrees $p$.
For the standard finite element method with full polynomial spaces, it is known that the pre-asymptotic regime improves as the polynomial degree increases; see, e.g., \cite{MS11}.
In view of the improved high-frequency behavior often observed for Trefftz-type discretizations, it is natural to ask whether a similar trend can also be seen for the embedded Trefftz DG method.

We therefore consider the polynomial degrees $p=2,3,4$ and several values of the wave number $\omega$, and plot the error against the number of degrees of freedom per wavelength $N_\lambda$, given by
\[
N_\lambda = \frac{2\pi \sqrt[d]{DOF}}{\omega\sqrt[d]{|\Omega|}}.
\]
This quantity provides a normalized measure of resolution, allowing results for different wave numbers to be compared on the same scale.
The wavelength $\lambda$ and the wave number $\omega$ are related by $\omega=2\pi/\lambda$, and we take $\omega=100,500,750,1000$.

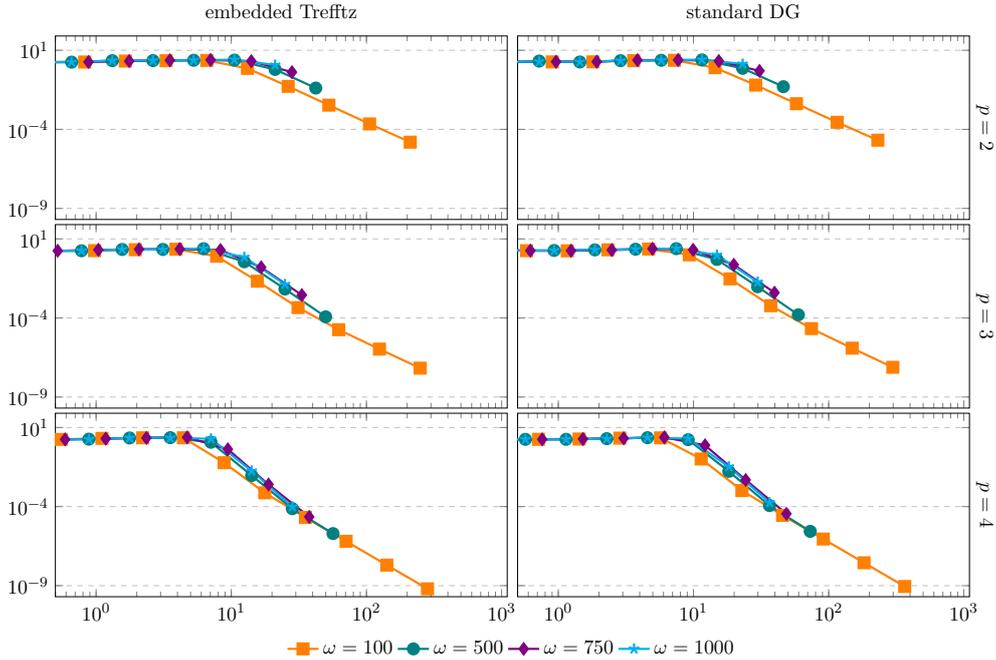
\begin{figure}[ht!]\centering
\resizebox{.9\linewidth}{!}{
\begin{tikzpicture}
	\begin{groupplot}[
	group style={
	group name={my plots},
	group size=2 by 3,
	horizontal sep=0.2cm,
    vertical sep=0.1cm,
    xticklabels at=edge bottom,
    yticklabels at=edge left
	},
    legend style={
    legend columns=6,
    at={(0.5,-0.2)},
    draw=none
    },
	ymajorgrids=true,
	grid style=dashed,
    cycle list name=paulcolors4,
    width=0.7\linewidth,
    height=0.7\linewidth/2,
    xmin=0.5,
    xmax=1100,
    ymin=2e-10,
    ymax=8e1,
	]      
    \def\omegas{100,500,750,1000}
    \nextgroupplot[ymode=log,xmode=log,title={embedded Trefftz}]
    \foreach \o in \omegas{
    \addplot+[discard if not={omega}{\o},discard if not={p}{2},discard if not={method}{etvol}] table [x=dofspwl, y=l2error, col sep=comma] {./numex/dofso2d.csv};
    }
    \nextgroupplot[ymode=log,xmode=log,title={standard DG}]
    \foreach \o in \omegas{
    \addplot+[discard if not={omega}{\o},discard if not={p}{2},discard if not={method}{dgvol}] table [x=dofspwl, y=l2error, col sep=comma] {./numex/dofso2d.csv};
    }
    \nextgroupplot[ymode=log,xmode=log]
    \foreach \o in \omegas{
    \addplot+[discard if not={omega}{\o},discard if not={p}{3},discard if not={method}{etvol}] table [x=dofspwl, y=l2error, col sep=comma] {./numex/dofso2d.csv};
    }
    \nextgroupplot[ymode=log,xmode=log]
    \foreach \o in \omegas{
    \addplot+[discard if not={omega}{\o},discard if not={p}{3},discard if not={method}{dgvol}] table [x=dofspwl, y=l2error, col sep=comma] {./numex/dofso2d.csv};
    }
    \nextgroupplot[ymode=log,xmode=log]
    \foreach \o in \omegas{
    \addplot+[discard if not={omega}{\o},discard if not={p}{4},discard if not={method}{etvol}] table [x=dofspwl, y=l2error, col sep=comma] {./numex/dofso2d.csv};
    }
    \nextgroupplot[ymode=log,xmode=log]
    \foreach \o in \omegas{
    \addplot+[discard if not={omega}{\o},discard if not={p}{4},discard if not={method}{dgvol}] table [x=dofspwl, y=l2error, col sep=comma] {./numex/dofso2d.csv};
    }
    \legend{$\omega=100$,$\omega=500$,$\omega=750$,$\omega=1000$}

	\end{groupplot}
    \node[right = 1cm of my plots c2r1.east,above=0cm of my plots c2r1.east,rotate=-90] {$p=2$};
    \node[right = 1cm of my plots c2r1.east,above=0cm of my plots c2r2.east,rotate=-90] {$p=3$};
    \node[right = 1cm of my plots c2r1.east,above=0cm of my plots c2r3.east,rotate=-90] {$p=4$};
\end{tikzpicture}}
\vspace{-0.5em}
\caption{
    Error against number of degrees of freedom per wavelength for different wave numbers $\omega$ and polynomial degrees $p$ for the problem with exact solution \eqref{eq:exp}.
}
\label{fig:exp}
\end{figure}

The results in \cref{fig:exp} are promising.
For both the embedded Trefftz and the standard DG discretization, the curves indicate that the pre-asymptotic region shifts to smaller values of $N_\lambda$ as $p$ increases.
In other words, higher-order discretizations appear to resolve the oscillatory solution accurately with fewer degrees of freedom per wavelength.
The same qualitative behavior is visible for the embedded Trefftz method, suggesting that the favorable $p$--dependence is retained also under the Trefftz constraint.

\subsection{Smoothly varying wave number}\label{sec:varo}

We consider the Helmholtz problem on the unit square $\Omega=(0,1)^2$ with wave number $\omega=5+\sin(x)+y^2$.
The exact solution is given by 
\begin{equation}\label{eq:varo}
    u(x,y)=\exp(i\omega x y).  
\end{equation}
The source term $f$ and the impedance boundary condition $g$ are chosen accordingly.
While this problem does not fall into the scope of our analysis, it provides a useful test of robustness, since the local approximation spaces are no longer aligned with a constant-coefficient Helmholtz operator.
Nevertheless, the results shown in \cref{fig:varo} indicate that the embedded Trefftz DG method still exhibits the expected convergence behavior.
In particular, the observed $L^2$- and DG-errors remain comparable to those of the standard DG method for the same polynomial degree, which suggests that the approach is not restricted to the constant-wave-number setting covered by the theory.

\begin{figure}[ht!]\centering
\resizebox{.9\linewidth}{!}{
\begin{tikzpicture}
	\begin{groupplot}[
	group style={
	group name={my plots},
	group size=2 by 1,
	horizontal sep=2cm,
	},
	ymajorgrids=true,
	grid style=dashed,
	]      
    \nextgroupplot[ymode=log,xmode=log,x dir=reverse, ylabel={$L^2$-error},xlabel={$h$}, cycle list name=paulcolors2, legend columns=2, legend pos=south west]
    \foreach \k in {3,4,5}{
    \addplot+[discard if not={p}{\k},discard if not={method}{etvol}] table [x=h, y=l2error, col sep=comma] {./numex/varo2d.csv};
    \addplot+[discard if not={p}{\k},discard if not={method}{dgvol}] table [x=h, y=l2error, col sep=comma] {./numex/varo2d.csv};
    }
    \addlegendimage{solid}
    \addplot[domain=0.03:0.5] {exp(-4*ln(1/x)-0.0)};
    \addplot[domain=0.03:0.5] {exp(-5*ln(1/x)-1.0)};
    \addplot[domain=0.03:0.5] {exp(-6*ln(1/x)-2.5)};
    \legend{$\IT^3$,$\IP^3$,$\IT^4$,$\IP^4$,$\IT^5$,$\IP^5$}%

    \nextgroupplot[ymode=log,xmode=log,x dir=reverse, ylabel={DG-error},xlabel={$h$}, cycle list name=paulcolors2, legend columns=2, legend pos=south west]
    \foreach \k in {3,4,5}{
    \addplot+[discard if not={p}{\k},discard if not={method}{etvol}] table [x=h, y=dgerror, col sep=comma] {./numex/varo2d.csv};
    \addplot+[discard if not={p}{\k},discard if not={method}{dgvol}] table [x=h, y=dgerror, col sep=comma] {./numex/varo2d.csv};
    }
    \addlegendimage{solid}
    \addplot[domain=0.03:0.5] {exp(-3*ln(1/x)+4.0)};
    \addplot[domain=0.03:0.5] {exp(-4*ln(1/x)+1.5)};
    \addplot[domain=0.03:0.5] {exp(-5*ln(1/x)+0.5)};
    \legend{$\IT^3$,$\IP^3$,$\IT^4$,$\IP^4$,$\IT^5$,$\IP^5$}%
	\end{groupplot}
\end{tikzpicture}}
\vspace{-0.5em}
\caption{
    Results for the problem with exact solution \eqref{eq:varo} with smoothly varying wave number. 
    The straight lines indicate the expected convergence rates of order $\mathcal O(h^{p+1})$ for the $L^2$-error and $\mathcal O(h^{p})$ for the DG-error.
}
\label{fig:varo}
\end{figure}
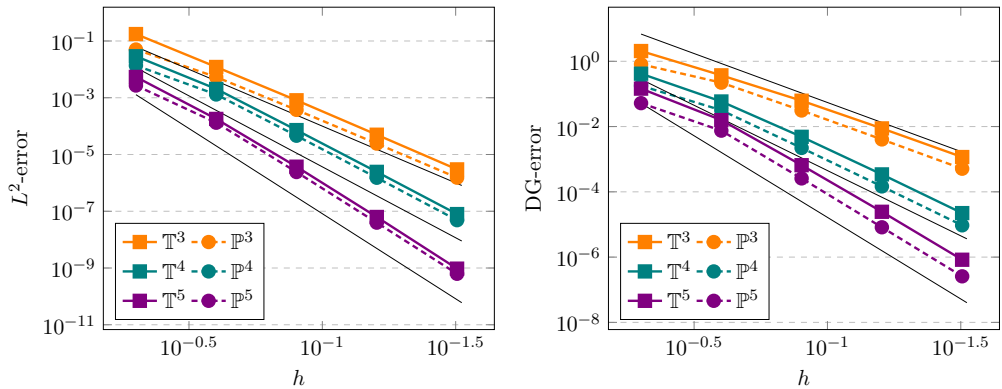

Overall, the numerical experiments confirm the main qualitative features predicted by the analysis: optimal algebraic convergence under $h$--refinement and rapid convergence under $p$--refinement.
At the same time, they indicate that the embedded Trefftz discretization can achieve accuracies comparable to those of the standard DG method with a significantly smaller number of degrees of freedom.
Finally, the last two tests suggest that the method remains effective also in regimes that are not yet covered by the present theory.

\section*{Acknowledgements}
\sloppy{
This research was funded in part by the Austrian Science Fund (FWF) \href{https://doi.org/10.55776/ESP4389824}{10.55776/ESP4389824}.
For open access purposes, the author has applied a CC BY public copyright license to any author-accepted manuscript version arising from this submission.}

\bibliographystyle{abbrvurl}
\bibliography{bib.bib}

\newpage
\section*{Appendix}

\begin{proof}[Proof of \cref{cor:cea}]
Let us denote
\[
\Bhnorm{v_h} = \bigl(\dgnorm{v_\IL}^2 + \dgnorm{v_\IT}^2\bigr)^{1/2}
\]
for $v_h = v_\IL + v_\IT \in V_h$.
For the analysis of \eqref{eq:PDEh}, we define for $u_h \in V_h$ the bilinear form
\begin{equation}\label{eq:Bh}
    \Bh (u_h, z_h ) := \sum_{K\in \Th} \langle \AK  u_h, q_K \rangle + a_h(u_h,v_h),
\end{equation}
where $z_h = (q_h, v_h) \in (\Qh \times \tch \ITh)$. We now show that $\Bh$ is $\tcb$-coercive
for the $\tcb$-coercivity map 
\[
    \tcb: V_h \to (\Qh \times \tch \ITh),\  u_h \mapsto  (\tcq u_h , \tct u_h ),
\]
which splits into $\tcq$ and $\tct$ that we define below. 

\paragraph{$\tct$-operator} We define $\tct:V_h\rightarrow \tch \ITh$ as
\begin{align}
    \tct u_h &:= 2\,\tch u_{\IT}  \in \tch \ITh,
    \addtocounter{equation}{1} \tag{$\tct\!|$\theequation}.
\end{align}
The continuity of $\tct$ follows immediately from the continuity of $\tch$.

Considering the $\tct$-part of $\Bh(u_h,\tcb u_h)$ we obtain
\begin{align}\label{eq:Bhstab1}
    \begin{split}
    \Re a_h(u_h,\tct u_h)
     &= 2\,\Re a_h(u_\IT + u_\IL, \tch u_\IT ) \\
     &= 2\,\Re a_h(u_\IT, \tch u_\IT) + 2\,\Re a_h(u_\IL, \tch u_\IT)  \\
     &\geq 2\,\Re a_h(u_\IT, \tch u_\IT)
     -2\,\CT (\Re a_h(u_\IT, \tch u_\IT))^{1/2} \norm{\ATh u_\IL}_{\Qh'} 
     \\
     &\geq \Re a_h(u_\IT, \tch u_\IT) - \CT^2 \norm{\ATh u_\IL}_{\Qh'}^2 .
    \end{split}
\end{align}

\paragraph{$\tcq$-operator} 
Let us now introduce the scaling
\[
\alpha=\CT^2+\frac{\ca}{\cA^2},
\]
Now we define $\tcq:V_h\rightarrow \Qh$ as
\begin{align}
    \tcq u_h := \sum_{K\in \Th} \tc_K u_h  \quad\text{with}\quad \tc_K u_h :=  \alpha\, R_K \AK u_\IL \in \QKh.
\end{align}
Here, $R_K:\QKh' \to \QKh$ denotes the Riesz lift, i.e.
\[
(R_K\varphi_h,q_h)_{\QKh}=\langle \varphi_h,q_h\rangle
\qquad\forall q_h\in\QKh,
\]
for $\varphi_h\in\QKh'$, $K\in\Th$.
\begin{align} \label{eq:tcTh-cont}
\norm{\tcq u_h}_{\Qh} & 
= \alpha \norm{\ATh u_\IL}_{\Qh'}
\quad \forall u_h \in V_h.
\end{align}
Now estimating the $\tcq$ part of $\Bh(u_h,\tcb u_h)$ we obtain 
\begin{eqs}\label{eq:Bhstab2}
\Re\langle \ATh u_h, \tcq u_h \rangle
&= \Re\sum_{K\in\Th}\alpha\,\langle \AK u_h, R_K \AK u_\IL\rangle
 = \Re\sum_{K\in\Th}\alpha\,\langle \AK u_\IL, R_K \AK u_\IL\rangle \\
&= \alpha \sum_{K\in\Th}\norm{\AK u_\IL}_{\QKh'}^2
\ge \alpha\cA^2\dgnorm{u_\IL}^2.
\end{eqs}

\paragraph{$\tcb$-coercivity} 
Summing up the two estimates \eqref{eq:Bhstab1} and \eqref{eq:Bhstab2} we obtain
\begin{eqs}\label{eq:Bhstab}
    \Re\Bh(u_h, \tcb u_h) 
    &\geq \Re a_h(u_\IT, \tch u_\IT) 
       - \CT^2 \norm{\ATh u_\IL}_{\Qh'}^2  
       + \alpha \norm{\ATh u_\IL}_{\Qh'}^2 \\
    &\geq  \Re a_h(u_\IT, \tch u_\IT) 
       + \frac{\ca}{\cA^2} \norm{\ATh u_\IL}_{\Qh'}^2 \\
    &\geq \ca(\dgnorm{u_\IT}^2 + \dgnorm{u_\IL}^2) .
\end{eqs}
Hence $\Bh$ is $\tcb$--coercive. 
For the first part of the theorem, it remains to show that $\tcb$ is bounded. Recalling the definition of $\tcb$ and the estimate \eqref{eq:tcTh-cont}, we have
\[
    \|\tcb u_h\|_{Z_h}^2
    = \norm{\tcq u_h}_{\Qh}^2 + \dgnorm{\tct u_h}^2
    \le 4\,\dgnorm{u_\IT}^2 + \CD^2\CA^2\alpha^2\dgnorm{u_\IL}^2
\]
for all $u_h\in V_h$.

Now let $u\in V$ be the solution of \eqref{eq:abstract} and $u_h\in V_h$ be the solution of \eqref{eq:PDEh}. Let $v_h\in V_h$ be arbitrary.
Set
\[
e_h:=v_h-u_h=e_\IL+e_\IT
\qquad\text{with}\qquad
e_\IL\in\ILh,\quad e_\IT\in\ITh.
\]
By \eqref{eq:Bhstab}, applied to $e_h$, we have
\begin{align*}
\ca \Bhnorm{e_h}^2
&\le \Re \Bh(e_h,\tcb e_h)
\le |\Bh(e_h,\tcb e_h)| \\
&\le |\Bh(v_h-u,\tcb e_h)| + |\Bh(u-u_h,\tcb e_h)| \\
&\le \norm{\ATh(v_h-u)}_{\Qh'}\norm{\tcq e_h}_{\Qh}
   + |a_h(v_h-u,\tct e_h)| \\
&\qquad
   + \norm{\ATh u-\ell_\Th}_{\Qh'}\norm{\tcq e_h}_{\Qh}
   + \norm{a_h(u,\cdot)-\ell_h}_{\tch\ITh'}\,\dgnorm{\tct e_h} \\
&\le \Bigl(\CA \tdgnorm{v_h-u} + \norm{\ATh u-\ell_\Th}_{\Qh'}\Bigr)\norm{\tcq e_h}_{\Qh} \\
&\qquad
   + \Bigl(\Ca \tdgnorm{v_h-u} + \norm{a_h(u,\cdot)-\ell_h}_{\tch\ITh'}\Bigr)\dgnorm{\tct e_h},
\end{align*}
where we used \eqref{eq:A_cont} and \eqref{eq:ah_cont}.

Moreover, by \eqref{eq:tcTh-cont} and \eqref{eq:Cstar},
\[
\norm{\tcq e_h}_{\Qh}
= \alpha \norm{\ATh e_\IL}_{\Qh'}
\le \alpha\CA\,\tdgnorm{e_\IL}
\le \alpha\CA\CD\,\dgnorm{e_\IL}
\le \alpha\CA\CD\,\Bhnorm{e_h},
\]
and by \eqref{eq:ah_coerc},
\[
\dgnorm{\tct e_h}
=2\,\dgnorm{\tch e_\IT}
\le 2\,\dgnorm{e_\IT}
\le 2\,\Bhnorm{e_h}.
\]

Therefore,
\begin{align*}
\ca \Bhnorm{e_h}
&\le
\CD\,\alpha\CA\Bigl(\CA \tdgnorm{v_h-u} + \norm{\ATh u-\ell_\Th}_{\Qh'}\Bigr) \\
&\qquad
+2\Bigl(\Ca \tdgnorm{v_h-u} + \norm{a_h(u,\cdot)-\ell_h}_{\tch\ITh'}\Bigr).
\end{align*}
Using $\dgnorm{e_h}\le \sqrt{2}\,\Bhnorm{e_h}\le 2\,\Bhnorm{e_h}$ and $\CD\ge 1$, we obtain
\begin{align*}
\dgnorm{v_h-u_h}
&\le
2\CD\frac{\alpha\CA^2}{\ca}\,\tdgnorm{v_h-u}
+2\CD\frac{\alpha\CA}{\ca}\,\norm{\ATh u-\ell_\Th}_{\Qh'}\\
&\qquad
+4\CD\frac{\Ca}{\ca}\,\tdgnorm{v_h-u}
+4\CD\frac{1}{\ca}\,\norm{a_h(u,\cdot)-\ell_h}_{\tch\ITh'}.
\end{align*}
Recalling
\[
\alpha=\CT^2+\frac{\ca}{\cA^2},
\]
this is exactly \eqref{eq:cea}.

\medskip
\noindent\emph{Remark.}
    The assumption \eqref{eq:Cstar} can be relaxed to $\tdgnorm{v_\IL}\le \CD\,\dgnorm{v_\IL}\qquad\forall v_\IL\in\ILh$, i.e.\ the norm equivalence only needs to hold on the subspace $\ILh$.
    This comes at the price of more complicated constants in the estimates.

\end{proof}

\begin{proof}
Let $e_h\in\ITh$ be defined by
\[
a_h(e_h,w_h)=a_h(u_h-u,w_h)\qquad\forall w_h\in\tch\ITh .
\]
Since $u_h$ solves \eqref{eq:PDEh}, $a_h(u_h-u,w_h)=\ell_h(w_h)-a_h(u,w_h)=-(a_h(u,\cdot)-\ell_h)(w_h)$.
Testing with $w_h=\tch e_h$ and using \eqref{eq:ah_coerc} yields
\[
\ca\,\dgnorm{e_h}^2\le \Re a_h(e_h,\tch e_h)
\le \norm{a_h(u,\cdot)-\ell_h}_{\tch\ITh'}\dgnorm{\tch e_h}
\]
hence, using $\dgnorm{\tch e_h}\le \dgnorm{e_h}$,
\begin{equation}\label{eq:eh_bound_short}
\dgnorm{e_h}\le \frac1{\ca}\,\norm{a_h(u,\cdot)-\ell_h}_{\tch\ITh'}.
\end{equation}
By \eqref{eq:IH_DG_dom}, $\norm{e_h}_{\IH}\le C_{\IH}\dgnorm{e_h}$, so
\begin{equation}\label{eq:eh_IH_bound_short}
\norm{e_h}_{\IH}\le \frac{C_{\IH}}{\ca}\,\norm{a_h(u,\cdot)-\ell_h}_{\tch\ITh'}.
\end{equation}
Set $e:=u_h-u-e_h\in\Vsh$. By construction,
\[
a_h(e,w_h)=a_h(u_h-u,w_h)-a_h(e_h,w_h)=0\qquad\forall w_h\in\tch\ITh.
\]

Let $z\in W$ solve the dual problem
\[
a(v,z)=(e,v)_{\IH}\qquad\forall v\in V.
\]
Then $a(\cdot,z)\in\IH'$ and
$\sup_{v\in V}\frac{|a(v,z)|}{\norm{v}_{\IH}}=\norm{e}_{\IH}$.
By \eqref{eq:a-bit-consistent-exp},
\[
\norm{e}_{\IH}^2=(e,e)_{\IH}=a(e,z)=a_h(e,z).
\]
For any $z_h\in\tch\ITh$ we have $a_h(e,z_h)=0$, hence
\[
\norm{e}_{\IH}^2=a_h(e,z-z_h).
\]
For arbitrary $w_h\in W_h$ we write $z-z_h=(w_h-z_h)+(z-w_h)$ and use \eqref{eq:a-bit-cont-exp}:
\begin{align*}
\norm{e}_{\IH}^2
&\le C_{\rm adj}\tdgnorm{e}\Big(\norm{w_h-z_h}_{W_h}+\norm{z-w_h}_{\Wsh}\Big)\\
&\le C_{\rm adj}\tdgnorm{e}\Big(\norm{z-z_h}_{W_h}+2\norm{z-w_h}_{\Wsh}\Big).
\end{align*}
Taking the infimum over $z_h\in\tch\ITh$ and $w_h\in W_h$ and using
\eqref{eq:H-reg-bound-exp}--\eqref{eq:H-reg-bound2-exp} gives
\[
\norm{e}_{\IH}\le C_{\rm adj}\,(C_{\rm reg}^{(1)}+2C_{\rm reg}^{(2)})\,h_H\,\tdgnorm{e}.
\]
Since $\tdgnorm{e}\le \tdgnorm{u_h-u}+\tdgnorm{e_h}$ and $e_h\in V_h$,
\[
\tdgnorm{e_h}\le \CD\,\dgnorm{e_h}\le \frac{\CD}{\ca}\,\norm{a_h(u,\cdot)-\ell_h}_{\tch\ITh'}
\quad\text{by }\eqref{eq:eh_bound_short}.
\]
Therefore
\begin{equation}\label{eq:e_IH_bound_short}
\norm{e}_{\IH}\le C_{\rm adj}\,(C_{\rm reg}^{(1)}+2C_{\rm reg}^{(2)})\,h_H
\Big(\tdgnorm{u_h-u}+\frac{\CD}{\ca}\norm{a_h(u,\cdot)-\ell_h}_{\tch\ITh'}\Big).
\end{equation}

Fix $v_h\in V_h$. By triangle inequality and norm equivalence on $V_h$,
\[
\tdgnorm{u_h-u}\le \tdgnorm{u_h-v_h}+\tdgnorm{v_h-u}
\le \CD\,\dgnorm{u_h-v_h}+\tdgnorm{u-v_h}.
\]
Moreover $\dgnorm{u_h-v_h}\le \dgnorm{u_h-u}+\dgnorm{u-v_h}\le \dgnorm{u_h-u}+\tdgnorm{u-v_h}$, hence
\[
\tdgnorm{u_h-u}\le \CD\,\dgnorm{u_h-u}+(1+\CD)\tdgnorm{u-v_h}.
\]
Using $\dgnorm{u_h-u}\le \dgnorm{v_h-u_h}+\dgnorm{u-v_h}\le \dgnorm{v_h-u_h}+\tdgnorm{u-v_h}$ and \eqref{eq:cea},
\[
\dgnorm{u_h-u}\le (1+C_{\cea}^{(1)})\tdgnorm{u-v_h}
+ C_{\cea}^{(2)}\norm{\ATh u-\ell_\Th}_{\Qh'}
+ C_{\cea}^{(3)}\norm{a_h(u,\cdot)-\ell_h}_{\tch\ITh'} .
\]
Insert this into the previous estimate and take $\inf_{v_h\in V_h}$ to get
\begin{align*}
\tdgnorm{u_h-u}\le\ &
(1+2\CD+\CD C_{\cea}^{(1)})\inf_{v_h\in V_h}\tdgnorm{u-v_h}\\
&+ \CD C_{\cea}^{(2)}\norm{\ATh u-\ell_\Th}_{\Qh'}
+ \CD C_{\cea}^{(3)}\norm{a_h(u,\cdot)-\ell_h}_{\tch\ITh'} .
\end{align*}

Finally, $\norm{u-u_h}_{\IH}\le \norm{e_h}_{\IH}+\norm{e}_{\IH}$ and \eqref{eq:eh_IH_bound_short}, \eqref{eq:e_IH_bound_short}
yield \eqref{eq:IH_bound_framework}.
\end{proof}

\end{document}